\documentclass[preprint,11pt]{elsarticle}
\usepackage{mathrsfs}
\usepackage{amsfonts}
\usepackage{amsmath}
\usepackage{stmaryrd}
\usepackage{amssymb}
\usepackage{amsthm}
\usepackage{mathrsfs}
\usepackage{amsfonts}
\usepackage{amscd}
\usepackage{indentfirst}
\usepackage{enumerate}
\usepackage{amsmath,amsfonts,amssymb,amsthm}
\usepackage{amsmath,amssymb,amsthm,amscd}
\usepackage{graphicx,mathrsfs}
\usepackage{appendix}

\newcommand{\R}{\mathbb{R}}

\newcommand{\N}{\mathbb{N}}
\renewcommand{\theequation}{\arabic{section}.\arabic{equation}}

\newtheorem{Thm}{Theorem}[section]
\newtheorem{Lem}[Thm]{Lemma}

\newtheorem{Prop}[Thm]{Proposition}

\newtheorem{Rem}[Thm]{Remark}

\begin{document}
\begin{frontmatter}

\title{Uniqueness of positive bound states with multi-bump for\\  nonlinear Schr\"odinger equations
}


\author{Daomin Cao}
\ead{dmcao@amt.ac.cn}

\author{Shuanglong Li}
\ead{lishuanglong@amss.ac.cn}

\author{Peng Luo}
\ead{luopeng@whu.edu.cn}

\address{Institute of Applied Mathematics, AMSS, The Chinese Academy of Sciences, Beijing 100190, PR China}

\begin{abstract}
We are concerned with the following nonlinear Schr\"odinger equation
\begin{equation*}
-\varepsilon^2\Delta u+ V(x)u=|u|^{p-2}u,~u\in H^1(\R^N),
\end{equation*}
where $N\geq 3$, $2<p<\frac{2N}{N-2}$. For $\varepsilon$ small enough and a class of $V(x)$, we show the uniqueness of positive multi-bump
solutions concentrating  at $k$ different critical points of $V(x)$ under certain assumptions on asymptotic behavior of $V(x)$ and its first derivatives near those points. The degeneracy of critical points is allowed in this paper.
\end{abstract}
\begin{keyword}
Nonlinear Schr\"odinger equations, Concentration, Uniqueness, Pohozaev identity
\end{keyword}
\end{frontmatter}



\section{Introduction}
\setcounter{equation}{0}
In this paper, we consider the uniqueness of concentrating solutions to the following nonlinear Schr\"odinger equations
\begin{equation}\label{1.1}
\begin{cases}
-\varepsilon^2\Delta u+ V(x)u=|u|^{p-2}u,~x\in \R^N, \\
u\in H^1(\R^N),
\end{cases}
\end{equation}
where $\varepsilon>0$ is a small parameter, $N\geq 3$, $2<p<\frac{2N}{N-2}$.

Problem \eqref{1.1} appears in the study of standing waves for the following nonlinear Schr\"odinger equations
\begin{equation}\label{1.2}
i\varepsilon\frac{\partial \varphi}{\partial t}=-\varepsilon^2\Delta_x\varphi +(V(x)+E)\varphi-|\varphi|^{p-1}\varphi,~(x,t)\in \R^N\times\R^+,
\end{equation}
with the form $\varphi(x,t)=e^{-iEt/\varepsilon}u(x)$, where $i$ is the imaginary unit and $\varepsilon$ is the Planck constant. Equation \eqref{1.2} is one of the most important problem in nonlinear optics and quantum physics. Especially, it describes the transition from quantum to classical mechanics as $\varepsilon\rightarrow 0$, so the study of solutions to \eqref{1.2} for small $\varepsilon$ has a crucial interest in physical.

In recent decades, there are a  number of results on the existence of the solutions for  problem \eqref{1.1}. In  \cite{Floer}, for the case $N=1, p=3$, Floer and Weinstein  obtained a solution concentrating at the global non-degenerate minimum point when $\varepsilon$ is small enough. Later, Oh \cite{Oh1,Oh2} generalized Floer-Weinstein's results to higher dimension with $2<p<2N/(N-2)$ and they obtained the existence of positive multi-bump solutions concentrating at any given set of nondegenerate critical points of $V(x)$ as $\varepsilon\rightarrow 0$. In fact, the results in \cite{Floer,Oh1,Oh2} seem  to rely in essential way on the nondegeneracy of the critical points. Also, the existence of a single-bump solution
concentrating at the  critical point of $V(x)$ which may be degenerate as $\varepsilon\rightarrow 0$
 was obtained by Ambrosetti, Badiale, Cingolani \cite{Ambrosetti}. These results were obtained by Lyapunov-Schmidt reduction.

On the other hand, by using variational methods, Rabinowitz \cite{Rabinowitz} proved the existence of a positive ground solution to \eqref{1.1} for small $\varepsilon$ if $V(x)\in C^1(\R^N)$ satisfies
\begin{equation*}
\liminf_{|x|\rightarrow \infty}V(x)>\inf_{\R^N}V(x)> 0.
\end{equation*}
Later, Wang \cite{Wang1} gave the accurate characterization of the concentration behavior for the positive ground state and bound state solutions. The solutions in \cite{Rabinowitz,Wang1} are mainly single-bump solutions. For the multi-bump solutions, del Pino and Felmer \cite{delpino1,delpino3} got the existence of such solutions concentrating at the  critical points of $V(x)$ under some local conditions for the potential $V(x)$. Here they assumed that $V(x)$ is locally H\"{o}lder continuous on $\R^{N}$ and satisfies
\begin{equation*}
\displaystyle \inf_{x\in \partial\Omega_{i}}V(x)>\displaystyle\inf_{x\in\Omega_{i}}V(x)> 0,~  i=1,\cdots, k,
\end{equation*}
where $\Omega_{1},\cdots,\Omega_{k}$ are $k$ disjoint bounded regions. And the results in \cite{delpino3} are true when the critical points of $V(x)$ are  degenerate.

There are also a lot of results concerning on the existence of multi-bump solutions for problems
similar to \eqref{1.1}. For
Dirichlet problems with a subcritical nonlinearity on bounded domains, the solutions concentrating
at one or a couple of points were obtained in \cite{Cao,Li,Noussair}; For
the case of a critical nonlinearity, the results on the existence of multi-bump solutions can be found in \cite{Bahri,Rey}.
 For other results concerning the existence of the solutions with  the concentration phenomena, one can refer to  \cite{Ambrosetti1,Ambrosetti2,Bonheure,Cao4,Cingolani,delpino2,Deng1,Grossi,Gui,Kang} and the references therein.

As far as we know, the results on the uniqueness of solutions which have the concentration phenomena  are few. In this aspect, the first result is the uniqueness of solutions concentrating at one point for
Dirichlet problems with critical nonlinearity on bounded domains given by Glangetas in \cite{Glangetas}.
Later, Cao and Heinz \cite{Cao3} gained the uniqueness of solutions to \eqref{1.1} which concentrate at the nondegenerate critical points of $V(x)$. The results in
\cite{Cao3,Glangetas} are obtained by using the topological degree. Recently, Deng, Lin and Yan  \cite{Deng} got the local uniqueness and periodicity for the  solutions with infinitely many bumps of the prescribed scalar curvature problem which involves the critical Sobolev exponent by the technique of Pohozaev identity.
 For more work concerning the uniqueness of  solutions with  the concentration phenomena, one can also refer to  \cite{Cao1,Guo}.

However, whether the multi-bump solutions of problem \eqref{1.1} concentrating at the degenerate critical points of $V(x)$ are unique is still unknown. In this paper, inspired by
\cite{Deng} we solve this problem partially by using the Pohozaev identity. To
be specific, if two families of multi-bump solutions to \eqref{1.1} concentrate at the same set of critical points of $V(x)$, then they can be written in the same form by the results of \cite{Cao3}; next, we can get the useful estimate and use the local Pohozaev identity
to show that the two solutions are in fact the same.
Thus our uniqueness are essentially in the local sense.

Let $U_{a}(x)$ be the unique positive solution (see \cite{Kwong}) of the following problem
\begin{equation}\label{1.3}
\begin{cases}
-\Delta u+V(a)u=|u|^{p-2}u,~ \mbox{in} ~\R^N,\\
u(0)=\displaystyle\max_{x\in \R^N}u(x), ~u(x)\in H^1(\R^N),
\end{cases}
\end{equation}
where $a$ is a given point in $\R^N$ and $V(a)>0$. Also it is well-known  that $U_a(x)$ is a radially symmetric decreasing function satisfying for $|\alpha|\leq 1$
\begin{equation}\label{1.4}
|D^{\alpha}U_{a}(x)|\exp\big(\sqrt{V(a)}|x|\big)|x|^{\frac{N-1}{2}}\leq C,
\end{equation}
where $C>0$ is some constant in \cite{Gidas}.

In our paper, we consider a class of $V(x)$ as follows:
\bigskip
\\
\emph{
\noindent\textup{\textbf{($V_1$):}} $V(x)$ is a bounded $C^1$ function satisfying $\displaystyle\inf_{x\in \R^N}V(x)>0$.\\
\textup{\textbf{($V_2$):}} $V(x)$ satisfies the following expansions:
\begin{equation}\label{1.5}
\begin{cases}
V(x)=V(a_j)+b_j|x-a_j|^{m}+O(|x-a_j|^{m+1}),~~x\in B_{\delta}(a_j),\\
\nabla V(x)=mb_j|x-a_j|^{m-2}(x-a_j)+O(|x-a_j|^m),~~x\in B_{\delta}(a_j),
\end{cases}
\end{equation}
where $\delta>0$ is a small constant, $V(a_j)>0$, $m>1$, $b_j\neq 0$ and $j=1,\cdots,k$.}
\smallskip

Next, we define
 $$H_{\varepsilon}=\left\{u(x) \in H^1(\R^N), \int_{\R^N}\big(\varepsilon^2|\nabla u(x) |^2+V(x)u^2(x)\big)\mathrm{d}x<\infty\right\},$$
and for any $u(x)\in H_{\varepsilon}$,
\begin{equation*}
\|u\|_{\varepsilon}=\big(u(x),u(x)\big)^{\frac{1}{2}}_{\varepsilon}=\big(\int_{\R^N}(\varepsilon^2|\nabla u(x)|^2+V(x)u^2(x))\mathrm{d}x\big)^{\frac{1}{2}}.
\end{equation*}
Then we have the following results:
\begin{Thm}\label{th1.1}
Let $\{u_\varepsilon^{(1)}(x)\}_{\varepsilon>0}, \{u_\varepsilon^{(2)}(x)\}_{\varepsilon>0}$ be two families of positive solutions of \eqref{1.1} concentrating at a set of $k$ different points $\{a_1,\cdots,a_k\}\subset \R^N$. Suppose that $V(x)$ satisfies \textup{($V_1$)} and \textup{($V_2$)}. Then for $\varepsilon$ small enough, $u_\varepsilon^{(1)}(x)\equiv u_\varepsilon^{(2)}(x)$
and
$u_{\varepsilon}(x)=u_\varepsilon^{(1)}(x)=u_\varepsilon^{(2)}(x)$ is of the form
\begin{equation}\label{1.6}
u_\varepsilon(x)=\sum_{j=1}^{k}U_{a_j}(\frac{x-x_{j,\varepsilon}}{\varepsilon})+w_\varepsilon(x),
\end{equation}
with $x_{j,\varepsilon},w_{\varepsilon}(x)$ satisfying,
for $j=1,\cdots,k$, as $\varepsilon\rightarrow 0$,
\begin{equation}\label{1.7}
|x_{j,\varepsilon}-a_j|=o(\varepsilon) ~\mbox{and}~
\|w_{\varepsilon}\|_\varepsilon=O(\varepsilon^{\frac{N}{2}+m}).
\end{equation}
\end{Thm}
\begin{Rem}\label{rem1.2}
Cao and Heinz \cite{Cao3} proved the uniqueness of solutions concentrating at the nondegenerate critical points of $V(x)$ by the topological degree. To using the topological degree in \cite{Cao3},
it is assumed  that the critical points set $\{a_1,\cdots,a_k\}$ of $V(x)$ are nondegenerate and
$V(x)$ is $C^2$ at $\{a_1,\cdots,a_k\}$.
However, Theorem \ref{th1.1} shows the uniqueness of solutions concentrating at the critical points of $V(x)$ which may be degenerate under the conditions \textup{($V_1$)} and \textup{($V_2$)}. We point out
that even under the same assumptions as in \cite{Cao3}, the proofs in the present paper are much simpler than those in \cite{Cao3}.

For the case $m=2$ in \eqref{1.5}, suppose that $\{a_1,\cdots,a_k\}$ are the nondegenerate critical points of $V(x)$, then Theorem \ref{th1.1} is the same as the results  in \cite{Cao3}.
However, the framework of using the topological degree in \cite{Cao3,Glangetas} does not work anymore
for the case $m\neq 2$ in \eqref{1.5}. Here, we use the Pohozaev identity to prove our main results.

Specifically, if $m>2$ in \eqref{1.5}, then $\{a_1,\cdots,a_k\}$ are the degenerate critical points of $V(x)$ and our results show the uniqueness of solutions concentrating at the degenerate critical points.

If $1<m<2$ in \eqref{1.5}, $V(x)$ is not $C^2$ at the critical points set $\{a_1,\cdots,a_k\}$.  And we can also obtain the uniqueness of solutions concentrating at $\{a_1,\cdots,a_k\}$.
There, it should point out that  we need new and careful estimates to hand this case.

These mean that our results  extend the results in \cite{Cao3} to more general cases which include  the degenerate case.
\end{Rem}

\begin{Rem}
Since problem \eqref{1.1} is the case of a subcritical nonlinearity and the positive solution of \eqref{1.3} can not be given explicitly, different arguments from \cite{Deng} should be applied to obtain
the estimates we need in our proof of Theorem \ref{th1.1}. Another point should be pointed out is that
the interaction between the bumps
must be taken into careful consideration.
\end{Rem}

\begin{Rem}
The role of  Pohozaev identity in the existence and nonexistence of solutions to problems with critical Sobolev exponents has been showed in many papers, see \cite{Brezis,Cao5,Cao6,Pohozaev} and the references therein. But the role of Pohozaev identity in the uniqueness is not well exploited until recently (see \cite{Deng,Guo}), we expect more applications of it in the further.
\end{Rem}


Our paper is organized as follows. In Section 2, we obtain some estimates which are essential to prove Theorem \ref{th1.1}. Next, by using the Pohozaev identity we give the detailed proof of Theorem \ref{th1.1} in  Section 3. Finally, we give estimates of some important quantities used repeatedly in this paper and their proofs   in the Appendix.

Throughout this paper, we will use the same $C$ to denote various generic positive constants, and we will
use $O(t)$, $o(t)$ to mean $|O(t)|\leq C|t|$, $o(t)/t\rightarrow 0$ as $t\rightarrow 0$. Finally, $o(1)$
denotes quantities that tend to $0$ as $\varepsilon\rightarrow 0$.
\section{Preliminaries}
\setcounter{equation}{0}
First, we define for any $a,y\in \R^N$
\begin{equation*}
\begin{split}
 E_{\varepsilon,a,y}&=\left \{ u(x)\in H^1(\R^N): \big(u(x),U_{a}(\frac{x-y}{\varepsilon})\big)_{\varepsilon}=0,~\big(u(x),\frac{\partial{U_{a}(\frac{x-y}{\varepsilon})}}{\partial{x_i}}\big)_{\varepsilon}=0, ~i=1,\cdots,N
 \right\}.
 \end{split}
\end{equation*}
Then, the following basic structure of the solutions concentrating at $k$ different points has been obtained by the Lyapunov-Schmidt reduction in
\cite{Cao3}.
\begin{Prop}\label{prop2.1}
If $\{u_{\varepsilon}(x)\}_{\varepsilon>0}$ is a family of  positive solutions of \eqref{1.1} concentrating at a set of $k$ different points $\{a_1,\cdots,a_k\}$, then $a_j(j=1,\cdots,k)$ must be a critical point of $V(x)$ and $u_{\varepsilon}(x)$ is of the form
\begin{equation}\label{2.1}
u_\varepsilon(x)=\sum_{j=1}^k(1+\alpha_{j,\varepsilon})
U_{a_j}(\frac{x-x_{j,\varepsilon}}{\varepsilon})+\upsilon_{\varepsilon}(x),
\end{equation}
with $x_{j,\varepsilon}$, $\upsilon_{\varepsilon} (x)\in \bigcap ^k_{j=1}E_{\varepsilon,a_j,x_{j,\varepsilon}}$ satisfying, for $j=1,\cdots,k$, as $\varepsilon\rightarrow 0$,
\begin{equation}\label{2.2}
|x_{j,\varepsilon}-a_j|=o(1),~
\alpha_{j,\varepsilon}=o(1),~
\|\upsilon_{\varepsilon}\|_{\varepsilon}=o(\varepsilon^{\frac{N}{2}}).
\end{equation}
\end{Prop}
\begin{proof}
See [\cite{Cao3} Theorem 1.1].
\end{proof}
\begin{Rem}\label{rem2.2}
Proposition \ref{prop2.1} shows that if a family of solutions has concentration phenomenon with multi-bump, then the solutions can be written in the form \eqref{2.1}. Also, letting
\begin{equation*}
w_{\varepsilon}(x)=\displaystyle\sum_{j=1}^k\alpha_{j,\varepsilon}U_{a_j}(\frac{x-x_{j,\varepsilon}}
{\varepsilon})+\upsilon_{\varepsilon} (x)
\end{equation*}
and using \eqref{2.2},
we can write the solution $u_{\varepsilon}(x)$ in Proposition \ref{prop2.1}
in the following form:
\begin{equation}\label{2.3}
u_\varepsilon(x)=\sum_{j=1}^kU_{a_j}(\frac{x-x_{j,\varepsilon}}{\varepsilon})+w_{\varepsilon}(x),
\end{equation}
with $x_{j,\varepsilon}$, $w_{\varepsilon} (x)$ satisfying, for $j=1,\cdots,k$, as $\varepsilon\rightarrow 0$,
\begin{equation}\label{2.4}
|x_{j,\varepsilon}-a_j|=o(1),~
\|w_{\varepsilon}\|_{\varepsilon}=o(\varepsilon^{\frac{N}{2}}).
\end{equation}
In this paper, for simplicity, we will use  \eqref{2.1} and \eqref{2.3} alternately.
\end{Rem}

Next, by the regularity theory of elliptic equations, $u_{\varepsilon}(x)$ is in fact  a classical solution. Then, we establish the Pohozaev identity which is crucial in our paper.
\begin{Prop}\label{Prop2.2}
Let $u(x)$ be the solution of \eqref{1.1}, then for any bounded open domain $\Omega$, we have the following Pohozaev identity
\begin{equation}\label{2.5}
\begin{split}
\int_{\Omega}\frac{\partial V(x)}{\partial x_i}u^2(x)\mathrm{d}x&=-2\varepsilon^2\int_{\partial\Omega}\frac{\partial u(x)}{\partial \nu}\frac{\partial u(x)}{\partial x_i}\mathrm{d}\sigma+\varepsilon^2\int_{\partial \Omega}|\nabla u(x)|^2\nu_i(x)\mathrm{d}\sigma\\&
~~~~+\int_{\partial\Omega}V(x)u^2(x)\nu_i(x)\mathrm{d}
\sigma-\frac{2}{p}\int_{\partial\Omega}|u(x)|^p\nu_i(x)\mathrm{d}\sigma,
\end{split}
\end{equation}
where $\nu(x)=\big(\nu_{1}(x),\cdots,\nu_N(x)\big)$ is the outward unit normal of $\partial \Omega$ and $i\in \{1,\cdots,N\}$.
\end{Prop}

\begin{proof}
Multiplying $\frac{\partial u(x)}{\partial x_i}$ on both sides of \eqref{1.1} and integrating on $\Omega$, we have
\begin{equation}\label{2.6}
-{\varepsilon}^2\int_{\Omega}\Delta u(x)\frac{\partial u(x)}{\partial x_i}\mathrm{d}x+\int_{\Omega}V(x)u(x)\frac{\partial u(x)}{\partial x_i}\mathrm{d}x=\int_{\Omega}\frac{\partial u(x)}{\partial x_i}|u(x)|^{p-2}u(x)\mathrm{d}x.
\end{equation}
Next,
\begin{equation}\label{2.7}
\begin{aligned}
\textrm{LHS of \eqref{2.6}}=&-\varepsilon^2\int_{\partial\Omega}\frac{\partial u(x)}{\partial x_i}\frac{\partial u(x)}{\partial \nu}\mathrm{d}\sigma+\varepsilon^2\int_{\Omega}\nabla u(x)\cdot\nabla\frac{\partial u(x)}{\partial x_i}\mathrm{d}x\\&+\frac{1}{2}\int_{\partial \Omega}u^2(x)V(x)\nu_i(x)\mathrm{d}\sigma-\frac{1}{2}\int_{\Omega}u^2(x)\frac{\partial V(x)}{\partial x_i}\mathrm{d}x,
\end{aligned}
\end{equation}
also,
\begin{equation}\label{2.8}
\int_{\Omega}\nabla u(x)\cdot\nabla\frac{\partial u(x)}{\partial x_i}\mathrm{d}x=\frac{1}{2}\int_{\partial\Omega}|\nabla u(x)|^2\nu_i(x)\mathrm{d}\sigma.
\end{equation}
On the other hand, by Green's formula, we have
\begin{equation}\label{2.9}
\textrm{RHS of \eqref{2.6}}=\frac{1}{p}\int_{\partial\Omega}|u(x)|^p\nu_i(x)\mathrm{d}\sigma.
\end{equation}
Then  \eqref{2.5} follows from \eqref{2.6}, \eqref{2.7}, \eqref{2.8} and \eqref{2.9}.
\end{proof}

In the rest of this section, we will show that the estimates of
$|x_{j,\varepsilon}-a_j|$, $\alpha_{j,\varepsilon}$ and $\|w_{\varepsilon}\|_{\varepsilon}$ in Proposition \ref{prop2.1} can
be improved step by step.
\begin{Lem}\label{lem2.3}
Let $u_{\varepsilon}(x)$ be the solution of \eqref{1.1} with the form \eqref{2.1}. Suppose  that \textup{($V_1$)} and \textup{($V_2$)} are satisfied, then we have
\begin{equation}\label{2.10}
|x_{j,\varepsilon}-a_j|=O(\varepsilon)+O(\sum^{k}_{l=1}\alpha_{l,\varepsilon}^{\frac{2}{m-1}}),~ j=1,\cdots,k.
\end{equation}
\end{Lem}
\begin{proof}
First, taking $u(x)=u_\varepsilon(x)=\displaystyle\sum_{l=1}^{k}U_{a_l}(\frac{x-x_{l,\varepsilon}}{\varepsilon})
+w_\varepsilon(x)$ and  $\Omega=B_{d}(x_{j,\varepsilon})$  for some small constant $d>0$ in the Pohozaev identity \eqref{2.5},  we have, for $i=1,\cdots,N$,
\begin{equation}\label{2.11}
\begin{aligned}
\int_{B_{d}(x_{j,\varepsilon})}\frac{\partial V(x)}{\partial x_i}\big[\sum_{l=1}^{k}U_{a_l}(\frac{x-x_{l,\varepsilon}}{\varepsilon})+w_\varepsilon(x)\big]^2\mathrm{d}x
=I_1+I_2+I_3,
\end{aligned}
\end{equation}
where
\begin{equation*}
I_1=-2\varepsilon^2\int_{\partial B_{d}(x_{j,\varepsilon})}\frac{\partial \big(\sum_{l=1}^{k}U_{a_l}(\frac{x-x_{l,\varepsilon}}{\varepsilon})+w_\varepsilon(x)\big)}{\partial \nu}\frac{\partial \big(\sum_{l=1}^{k}U_{a_l}(\frac{x-x_{l,\varepsilon}}{\varepsilon})+w_\varepsilon(x)\big)}{\partial x_i}\mathrm{d}\sigma,
\end{equation*}
\begin{equation*}
\begin{split}
I_2&=\int_{\partial B_{d}(x_{j,\varepsilon})}\big[\varepsilon^2\big|\nabla \big(\sum_{l=1}^{k}U_{a_l}(\frac{x-x_{l,\varepsilon}}{\varepsilon})+w_\varepsilon(x)\big)\big|^2
+V(x)\big(\sum_{l=1}^{k}U_{a_l}(\frac{x-x_{l,\varepsilon}}
{\varepsilon})+w_\varepsilon(x)\big)^2\big]
\nu_i(x)\mathrm{d}\sigma,
\end{split}
\end{equation*}
and
\begin{equation*}
I_3=-\frac{2}{p}\int_{\partial B_{d}(x_{j,\varepsilon})}\big|\big(
\sum_{l=1}^{k}U_{a_l}(\frac{x-x_{l,\varepsilon}}{\varepsilon})+w_\varepsilon(x)\big)\big|
^p\nu_i(x)\mathrm{d}\sigma.
\end{equation*}
Now, using \eqref{A.1} and \eqref{A.3} in the Appendix, we have, for any $\gamma>0$,
\begin{equation}\label{2.12}
\begin{aligned}
\int_{B_{d}(x_{j,\varepsilon})}&\frac{\partial V(x)}{\partial x_i}\big[\sum_{l=1}^{k}U_{a_l}(\frac{x-x_{l,\varepsilon}}{\varepsilon})+w_\varepsilon(x)\big]^2\mathrm{d}x
\\& = \int_{B_{d}(x_{j,\varepsilon})}\frac{\partial V(x)}{\partial x_i}\big[\sum_{l=1}^{k}U_{a_l}^2\big(\frac{x-x_{l,\varepsilon}}{\varepsilon}\big)
+w^2_{\varepsilon}(x)\big]\mathrm{d}x
+O(\varepsilon^{\gamma})\\&
~~~~+2\displaystyle\sum_{l=1}^{k}\int_{B_{d}(x_{j,\varepsilon})}\frac{\partial V(x)}{\partial x_i}U_{a_l}\big(\frac{x-x_{l,\varepsilon}}{\varepsilon}\big)w_{\varepsilon}(x)\mathrm{d}x
\\&=
\int_{B_{d}(x_{j,\varepsilon})}\frac{\partial V(x)}{\partial x_i}U_{a_j}^2\big(\frac{x-x_{j,\varepsilon}}{\varepsilon}\big)\mathrm{d}x
+O(\|w_{\varepsilon}\|^2_{\varepsilon}+\varepsilon^{\gamma})\\&
~~~~+2\int_{B_{d}(x_{j,\varepsilon})}\frac{\partial V(x)}{\partial x_i}U_{a_j}\big(\frac{x-x_{j,\varepsilon}}{\varepsilon}\big)w_{\varepsilon}(x)\mathrm{d}x.
\end{aligned}
\end{equation}
So, by choosing $\gamma>0$ appropriately, from \eqref{2.12} and \eqref{A.6}, we have
\begin{equation}\label{2.13}
\begin{aligned}
&\int_{B_{d}(x_{j,\varepsilon})}\frac{\partial V(x)}{\partial x_i}\big[\sum_{l=1}^{k}U_{a_l}(\frac{x-x_{l,\varepsilon}}{\varepsilon})+w_\varepsilon(x)\big]^2\mathrm{d}x
\\&=\int_{B_{d}(x_{j,\varepsilon})}\frac{\partial V(x)}{\partial x_i}U_{a_j}^2\big(\frac{x-x_{j,\varepsilon}}{\varepsilon}\big)\mathrm{d}x+
O\big(\|w_{\varepsilon}\|^2_{\varepsilon}+\varepsilon^{N+2m-2}+
\varepsilon^N|x_{j,\varepsilon}-a_j|^{2m-2}\big).
\end{aligned}
\end{equation}
Next, from \eqref{A.2} and Lemma \ref{lem-A-4}, we know, for any $\gamma>0$,
\begin{equation}\label{2.14}
I_1=O\big(\|w_{\varepsilon}\|^2_{\varepsilon}+\varepsilon^{\gamma}\big), ~
I_2=O\big(\|w_{\varepsilon}\|^2_{\varepsilon}+\varepsilon^{\gamma}\big)~\mbox{and}~
I_3=O\big(\int_{\partial B_{d}(x_{j,\varepsilon})}|w_{\varepsilon}(x)|^p\mathrm{d}\sigma+\varepsilon^{\gamma}\big).
\end{equation}
Also,  from \eqref{2.4}, \eqref{A.11} and Lemma \ref{lem-A-4}, we obtain
\begin{equation}\label{2.15}
\int_{\partial B_{d}(x_{j,\varepsilon})}|w_{\varepsilon}(x)|^p\mathrm{d}\sigma\leq
C\int_{\R^N}|w_{\varepsilon}(x)|^p\mathrm{d}x \leq C \|w_{\varepsilon}\|^2_{\varepsilon}.
\end{equation}
Then combining \eqref{2.14} and \eqref{2.15}, we see that for any  $\gamma>0$,
\begin{equation}\label{2.16}
I_1+I_2+I_3=O\big(\|w_{\varepsilon}\|^2_{\varepsilon}+\varepsilon^{\gamma}\big).
\end{equation}
From \eqref{2.4}, \eqref{2.11}, \eqref{2.13}, \eqref{2.16} and \eqref{B.3}, taking $\gamma$ appropriately,
for $i=1,\cdots,N$, we have
\begin{equation}\label{2.17}
\begin{split}
\displaystyle \int_{B_{d}(x_{j,\varepsilon})}\frac{\partial V(x)}{\partial x_i}U_{a_j}^2\big(\frac{x-x_{j,\varepsilon}}{\varepsilon}\big)\mathrm{d}x&=
O\big(\varepsilon^{N}(\varepsilon^{2m-2}+\max_{l=1,\cdots,k}|x_{l,\varepsilon}-a_l|^{2m-2}+
\sum^N_{l=1}\alpha_{l,\varepsilon}^2)\big).
\end{split}
\end{equation}
On the other hand,
\begin{equation}\label{2.18}
\begin{aligned}
\int_{B_{d}(x_{j,\varepsilon})} &\frac{\partial V(x)}{\partial x_i}U_{a_j}^2\big(\frac{x-x_{j,\varepsilon}}{\varepsilon}\big)\mathrm{d}x
 \\&
=b_jm\varepsilon^{N}\int_{B_{\frac{d}{\varepsilon}}(0)}|\varepsilon y+x_{j,\varepsilon}-a_j|^{m-2}\cdot(\varepsilon y_i+x_{j,\varepsilon,i}-a_{j,i})U_{a_j}^2(y)\mathrm{d}y\\
&~~~~+O\big(\varepsilon^N\int_{B_{\frac{d}{\varepsilon}}(0)}|\varepsilon y+x_{j,\varepsilon}-a_j|^mU_{a_j}^2(y)\mathrm{d}y\big),
\end{aligned}
\end{equation}
where $y_i$, $x_{j,\varepsilon,i}$, $a_{j,i}$ are the $i$-th components of $y$, $x_{j,\varepsilon}$, $a_{j}$ for $i=1,\cdots,N$. And
\begin{equation}\label{2.19}
\int_{B_{\frac{d}{\varepsilon}}(0)}|\varepsilon y+x_{j,\varepsilon}-a_j|^mU_{a_j}^2(y)\mathrm{d}y=O(\varepsilon^{m}+|x_{j,\varepsilon}-a_j|^m).
\end{equation}
Then for $i=1,\cdots,N$, \eqref{2.17}, \eqref{2.18} and \eqref{2.19} imply
\begin{equation}\label{2.20}
\begin{split}
\big|\int_{B_{\frac{d}{\varepsilon}}(0)}&|\varepsilon y+x_{j,\varepsilon}-a_j|^{m-2}\cdot(\varepsilon y_i+x_{j,\varepsilon,i}-a_{j,i})U_{a_j}^2(y)\mathrm{d}y\big|\\&
\leq
C\big(\varepsilon^{2m-2}+\varepsilon^{m}+|x_{j,\varepsilon}-a_j|^m+
\max_{l=1,\cdots,k}|x_{l,\varepsilon}-a_l|^{2m-2}+
\sum^N_{l=1}\alpha_{l,\varepsilon}^2\big).
\end{split}
\end{equation}
Also, we have the following inequality:
\begin{equation}\label{2.21}
\big||a+b|^m-|a|^m-m|a|^{m-2}a\cdot b\big|\leq C\big(|a|^{m-m^*}|b|^{m^*}+b^{m}\big),
\end{equation}
where $a=(a_1,\cdots,a_N)\in \R^N$, $b=(b_1,\cdots,b_N)\in \R^N$,
$a\cdot b=\sum^{N}_{j=1}a_jb_j$, $m>1$, $m^*=\min\{m,2\}$ and the constant $C$ is independent of $a,b$.

Letting $a=\varepsilon y+x_{j,\varepsilon}-a_j$ and $b=-\varepsilon y$ in \eqref{2.21}, we can get
\begin{equation}\label{2.22}
\begin{split}
&\sum^{N}_{i=1} m|\varepsilon y+x_{j,\varepsilon}-a_j|^{m-2}
(x_{j,\varepsilon,i}-a_{j,i})(\varepsilon y_{i}+x_{j,\varepsilon,i}-a_{j,i}) \\&
\geq |x_{j,\varepsilon}-a_j|^m +(m-1)|\varepsilon y+x_{j,\varepsilon}-a_j|^m-
C(|\varepsilon y|^{m}+|\varepsilon y|^{m^*}|x_{j,\varepsilon}-a_{j}|^{m-m^*})\\&
\geq
|x_{j,\varepsilon}-a_j|^m-C(|\varepsilon y|^{m}+|\varepsilon y|^{m^*}|x_{j,\varepsilon}-a_{j}|^{m-m^*}).
\end{split}
\end{equation}
Then multiplying $U_{a_j}^2(y)$ on both sides of \eqref{2.22} and integrating on $B_{\frac{d}{\varepsilon}}(0)$, we obtain
\begin{equation}\label{2.23}
\begin{split}
&|x_{j,\varepsilon}-a_{j}|\sum^{N}_{i=1}\big|\int_{B_{\frac{d}{\varepsilon}}(0)}|\varepsilon y+x_{j,\varepsilon}-a_j|^{m-2}
(\varepsilon y_{i}+x_{j,\varepsilon,i}-a_{j,i})U_{a_j}^2(y)\mathrm{d}y\big|\\&
\geq
\sum^{N}_{i=1}\big[(x_{j,\varepsilon,i}-a_{j,i})\int_{B_{\frac{d}{\varepsilon}}(0)}|\varepsilon y+x_{j,\varepsilon}-a_j|^{m-2}
(\varepsilon y_{i}+x_{j,\varepsilon,i}-a_{j,i})U_{a_j}^2(y)\mathrm{d}y\big]\\&
\geq
 \frac{1}{m}|x_{j,\varepsilon}-a_j|^m\int_{B_{\frac{d}{\varepsilon}}(0)}U_{a_j}^2(y)\mathrm{d}y-
 \frac{C}{m}\int_{B_{\frac{d}{\varepsilon}}(0)}\big(|\varepsilon y|^{m}+|\varepsilon y|^{m^*}|x_{j,\varepsilon}-a_{j}|^{m-m^*}\big)U_{a_j}^2(y)\mathrm{d}y.
\end{split}
\end{equation}
So \eqref{2.20} and \eqref{2.23} imply
\begin{equation}\label{2.24}
\begin{split}
|x_{j,\varepsilon}-a_j|^m&
\leq C\big[|x_{j,\varepsilon}-a_j|\big(\varepsilon^{2m-2}+\varepsilon^{m}+
\sum^N_{l=1}\alpha_{l,\varepsilon}^2\big)+
\max_{l=1,\cdots,k}|x_{l,\varepsilon}-a_l|^{2m-1}\\&
~~~~+|x_{j,\varepsilon}-a_j|^{m+1}+\varepsilon^{m}+|x_{j,\varepsilon}-a_{j}|^{m-m^*}\varepsilon^{m^*}\big].
\end{split}
\end{equation}
Also, by H\"older's inequality, for any $\eta>0$, we know
\begin{equation}\label{2.25}
\begin{split}
|x_{j,\varepsilon}-a_j|&\varepsilon^{2m-2}+|x_{j,\varepsilon}-a_j|
\sum^N_{l=1}\alpha_{l,\varepsilon}^2
+|x_{j,\varepsilon}-a_{j}|^{m-m^*}\varepsilon^{m^*}\\&
\leq \eta|x_{j,\varepsilon}-a_j|^m+C_{\eta}\big(\varepsilon^{2m}
+\sum^N_{l=1}\alpha_{l,\varepsilon}^{\frac{2m}{m-1}}+\varepsilon^{m}\big).
\end{split}
\end{equation}
Then combining \eqref{2.4}, \eqref{2.24} and \eqref{2.25}, we have
\begin{equation*}
\begin{split}
|x_{j,\varepsilon}-a_j|^m
\leq C\big(\varepsilon^m+
\sum^N_{l=1}\alpha_{l,\varepsilon}^{\frac{2m}{m-1}}\big).
\end{split}
\end{equation*}
This means that \eqref{2.10} is true.
\end{proof}

Here we need  \eqref{2.21} to handle the term $|\varepsilon y+x_{j,\varepsilon}-a_j|^{m-2}
(\varepsilon y_{i}+x_{j,\varepsilon,i}-a_{j,i})$ in \eqref{2.20}  because the index $m$ may be less than $2$, which is quite different from the technique in \cite{Deng}.

Next,
Lemma \ref{lem2.3} and Proposition \ref{prop-A.2} imply
  \begin{equation}\label{2.26}
    \|w_{\varepsilon}\|^2_{\varepsilon}=O(\varepsilon^{N+2m})+O(\varepsilon^{N}
    \sum^{k}_{j=1}\alpha_{j,\varepsilon}^2),~~
    \|\upsilon_{\varepsilon}\|^2_{\varepsilon}=O(\varepsilon^{N+2m})+O(\varepsilon^{N}\sum^{k}_{j=1}\alpha_{j,\varepsilon}^2).
  \end{equation}
\begin{Prop}\label{prop2.4}
Let $u_{\varepsilon}(x)$ be the solution of \eqref{1.1} with the form \eqref{2.1}, suppose that \textup{($V_1$)} and \textup{($V_2$)} are satisfied, then we have
\begin{equation}\label{2.27}
\alpha_{j,\varepsilon}=O(\varepsilon^m),~j=1,\cdots,k.
\end{equation}
\end{Prop}
\begin{proof}
From Proposition \ref{prop2.1}, let
\begin{equation*}
u_\varepsilon(x)=\sum_{j=1}^k(1+\alpha_{j,\varepsilon})U_{a_j}(\frac{x-x_{j,\varepsilon}}
{\varepsilon})+\upsilon_{\varepsilon}(x)
\end{equation*}
be a positive solution of \eqref{1.1} concentrating at $\{a_1,\cdots,a_k\}$, then
\begin{equation}\label{2.28}
\|u_{\varepsilon}\|^2_{\varepsilon}=\int_{\R^N}u_{\varepsilon}^p(x)\mathrm{d}x.
\end{equation}
Now we set
\begin{equation}\label{2.29}
\bar{u}_{\varepsilon}(x)=\frac{u_{\varepsilon}(x)}{1+\alpha_{1,\varepsilon}}=
\sum_{j=1}^k\beta_{j,\varepsilon}U_{a_j}(\frac{x-x_{j,\varepsilon}}
{\varepsilon})+\beta_{\varepsilon}\upsilon_{\varepsilon}(x),
\end{equation}
where
\begin{equation}\label{2.30}
\beta_{j,\varepsilon}=\frac{1+\alpha_{j,\varepsilon}}{1+\alpha_{1,\varepsilon}},~
\beta_{\varepsilon}=\frac{1}{1+\alpha_{1,\varepsilon}},~j=1,\cdots,k.
\end{equation}
Then \eqref{2.2} and \eqref{2.30} implies
\begin{equation}\label{2.31}
\beta_{1,\varepsilon}=1,~\beta_{j,\varepsilon}=1+o(1), ~\beta_{\varepsilon}=1+o(1),~j=2,\cdots,k.
\end{equation}
Next, \eqref{2.28} and \eqref{2.29} show that
\begin{equation}\label{2.32}
(1+\alpha_{1,\varepsilon})^{p-2}=
\frac{\|\bar{u}_{\varepsilon}\|^2_{\varepsilon}}{\int_{\R^N}\bar{u}_{\varepsilon}^p(x)\mathrm{d}x}.
\end{equation}

On the other hand, letting
\begin{equation*}
K_{\varepsilon}(u)=\frac{\|u\|^2_{\varepsilon}}
{\big(\int_{\R^N}u^p(x)\mathrm{d}x\big)^{\frac{2}{p}}},
\end{equation*}
 by Lagrange multiplier method we can verify that $\bar{u}_{\varepsilon}(x)$ is a critical point
of the functional $K_{\varepsilon}(u)$. Then from $DK_{\varepsilon}\big(\bar{u}_{\varepsilon}(x)\big)\big(U_{a_1}(\frac{x-x_{1,\varepsilon}}{\varepsilon})\big)=0$,
using the fact that $\upsilon_{\varepsilon}(x)$ and $U_{a_j}(\frac{x-x_{j,\varepsilon}}
{\varepsilon})$ are  orthogonal, we get
\begin{equation}\label{2.33}
\frac{\|\bar{u}_{\varepsilon}\|^2_{\varepsilon}}{\int_{\R^N}\bar{u}_{\varepsilon}^p(x)\mathrm{d}x}
=
\frac{\big(\sum_{j=1}^k\beta_{j,\varepsilon}U_{a_j}(\frac{x-x_{j,\varepsilon}}
{\varepsilon}), U_{a_1}(\frac{x-x_{1,\varepsilon}}
{\varepsilon})\big)_{\varepsilon}}{\int_{\R^N}|\bar{u}_{\varepsilon}(x)|^{p-2}\bar{u}_{\varepsilon}(x)U_{a_1}
(\frac{x-x_{1,\varepsilon}}{\varepsilon})\mathrm{d}x}.
\end{equation}
It is not difficult to show the following inequality:
  \begin{equation}\label{2.34}
  \begin{split}
 \big|(\sum^{k+1}_{i=1}d_i)^{l}-d_1^{l}\big|&\leq C\big(d_1^{l-1}|\sum^{k+1}_{i=2}d_i|+\sum^{k+1}_{i=2}|d_i|^{l
 }\big),
  \end{split}
 \end{equation}
where $l>1$ and $d_i\in\R$.

Then from the fact that $\beta_{1,\varepsilon}=1$ in \eqref{2.31} and \eqref{2.34}, we have
 \begin{equation}\label{2.35}
\begin{split}
|\bar{u}_{\varepsilon}(x)|^{p-2}\bar{u}_{\varepsilon}(x)&
=U^{p-1}_{a_1}
(\frac{x-x_{1,\varepsilon}}{\varepsilon})+O\big(\sum_{j=2}^kU^{p-1}_{a_j}(\frac{x-x_{j,\varepsilon}}
{\varepsilon})+|\upsilon_{\varepsilon}(x)|^{p-1}\big)
\\&
~~~~+ O\big(U^{p-2}_{a_1}
(\frac{x-x_{1,\varepsilon}}{\varepsilon})(\sum_{j=2}^kU_{a_j}(\frac{x-x_{j,\varepsilon}}
{\varepsilon})+|\upsilon_{\varepsilon}(x)|)\big).
\end{split}
\end{equation}
Then combining \eqref{2.35} and \eqref{A.3}, we can deduce
\begin{equation}\label{2.36}
\begin{split}
\int_{\R^N}&|\bar{u}_{\varepsilon}(x)|^{p-2}\bar{u}_{\varepsilon}(x)U_{a_1}
(\frac{x-x_{1,\varepsilon}}{\varepsilon})\mathrm{d}x\\&
=
\int_{\R^N}U^p_{a_1}(\frac{x-x_{1,\varepsilon}}{\varepsilon})
\mathrm{d}x+O\big(\int_{\R^N}U^{p-1}_{a_1}
(\frac{x-x_{1,\varepsilon}}{\varepsilon})|\upsilon_{\varepsilon}(x)|\mathrm{d}x\big)\\&
~~~~+O\big(\int_{\R^N}|\upsilon_{\varepsilon}(x)|^{p-1}U_{a_1}
(\frac{x-x_{1,\varepsilon}}{\varepsilon})\mathrm{d}x\big)+O(\varepsilon^\gamma).
\end{split}
\end{equation}
Next, \eqref{1.3}, \eqref{2.10}, \eqref{2.26} and \eqref{A.5} imply
\begin{equation}\label{2.37}
\begin{split}
\int_{\R^N}U^{p-1}_{a_1}
(\frac{x-x_{1,\varepsilon}}{\varepsilon})|\upsilon_{\varepsilon}(x)|\mathrm{d}x&
=\int_{\R^N}\big(V(a_1)-V(x)\big)U_{a_1}\big(\frac{x-x_{1,\varepsilon}}{\varepsilon}\big)
|\upsilon_{\varepsilon}(x)|\mathrm{d}x\\&
=O(\varepsilon^{N+2m}
+\varepsilon^{N}\sum_{j=1}^k\alpha^2_{j,\varepsilon}).
\end{split}
\end{equation}
Also, by H\"older's inequality and \eqref{2.26}, we have
\begin{equation}\label{2.38}
\begin{split}
\int_{\R^N}&|\upsilon_{\varepsilon}(x)|^{p-1}U_{a_1}
(\frac{x-x_{1,\varepsilon}}{\varepsilon})\mathrm{d}x
\\&\leq \big(\int_{\R^N}|\upsilon_{\varepsilon}(x)|^{2^*}\mathrm{d}x\big)^{\frac{p-1}{2^*}}
\cdot \big(\int_{\R^N}U^{\frac{2^*}{2^*-p+1}}_{a_1}
(\frac{x-x_{1,\varepsilon}}{\varepsilon})\mathrm{d}x\big)^{1-\frac{p-1}{2^*}}
\\&
\leq C\big(\varepsilon^{-1}\|\upsilon_{\varepsilon}\|_{\varepsilon}\big)^{p-1}\cdot
 \varepsilon^{N-\frac{(N-2)(p-1)}{2}}\\&
\leq C\big(\varepsilon^{N+m(p-1)}+\varepsilon^{N}\sum_{j=1}^k\alpha^{p-1}_{j,\varepsilon}\big).
\end{split}
\end{equation}
Letting $p^*=\min\{p,3\}$, from \eqref{2.36}, \eqref{2.37} and \eqref{2.38}, we see
\begin{equation}\label{2.39}
\begin{split}
\int_{\R^N}&|\bar{u}_{\varepsilon}(x)|^{p-2}\bar{u}_{\varepsilon}(x)U_{a_1}
(\frac{x-x_{1,\varepsilon}}{\varepsilon})\mathrm{d}x\\&=
\int_{\R^N}U^p_{a_1}(\frac{x-x_{1,\varepsilon}}{\varepsilon})
\mathrm{d}x+O\big[\varepsilon^{N}(\varepsilon^{m(p^*-1)}+\sum_{j=1}^k\alpha^{p^*-1}_{j,\varepsilon})\big]\\&
=\varepsilon^{N}\big[\int_{\R^N}U^p_{a_1}
(x)\mathrm{d}x+O(\varepsilon^{m(p^*-1)}+\sum_{j=1}^k\alpha^{p^*-1}_{j,\varepsilon})\big].
\end{split}
\end{equation}
Also from  the fact that $\beta_{1,\varepsilon}=1$ in \eqref{2.31} and  \eqref{A.3}, we know
\begin{equation}\label{2.40}
\begin{split}
\big(\sum_{j=1}^k\beta_{j,\varepsilon}U_{a_j}(\frac{x-x_{j,\varepsilon}}
{\varepsilon}), U_{a_1}(\frac{x-x_{1,\varepsilon}}
{\varepsilon})\big)_{\varepsilon}&
=\big(U_{a_1}(\frac{x-x_{1,\varepsilon}}
{\varepsilon}), U_{a_1}(\frac{x-x_{1,\varepsilon}}
{\varepsilon})\big)_{\varepsilon}+O(\varepsilon^{\gamma})\\&
=\int_{\R^N}\big(V(x)-V(a_1)\big)U_{a_1}^2
\big(\frac{x-x_{1,\varepsilon}}{\varepsilon}\big)\mathrm{d}x\\&
~~~~+\int_{\R^N}U^p_{a_1}(\frac{x-x_{1,\varepsilon}}{\varepsilon})
\mathrm{d}x+O(\varepsilon^{\gamma}).
\end{split}
\end{equation}
So, from  \eqref{2.10}, \eqref{2.40} and \eqref{A.5}, taking suitable $\gamma>0$, we obtain
\begin{equation}\label{2.41}
\big(\sum_{j=1}^k\beta_{j,\varepsilon}U_{a_j}(\frac{x-x_{j,\varepsilon}}
{\varepsilon}), U_{a_1}(\frac{x-x_{1,\varepsilon}}
{\varepsilon})\big)_{\varepsilon}=\varepsilon^{N}\int_{\R^N}U^p_{a_1}
(x)\mathrm{d}x+O\big(\varepsilon^{N}\sum_{j=1}^k\alpha^{2}_{j,\varepsilon}+
\varepsilon^{N+m}\big).
\end{equation}
Then combining \eqref{2.32}, \eqref{2.33}, \eqref{2.39} and \eqref{2.41}, we have
\begin{equation*}
\begin{split}
(1+\alpha_{1,\varepsilon})^{p-2}&=\frac{
\int_{\R^N}U^p_{a_1}
(x)\mathrm{d}x+O\big(\sum_{j=1}^k\alpha^{2}_{j,\varepsilon}\big)+
O(\varepsilon^{m})}{\int_{\R^N}U^p_{a_1}
(x)\mathrm{d}x+O(\varepsilon^{m(p^*-1)}+\sum_{j=1}^k\alpha^{p^*-1}_{j,\varepsilon})}\\&
=1+O\big(\displaystyle\sum_{j=1}^k\alpha^{p^*-1}_{j,\varepsilon}\big)+
O(\varepsilon^{m}).
\end{split}
\end{equation*}
Similar to the above procedure, we can get
\begin{equation}\label{2.42}
(1+\alpha_{i,\varepsilon})^{p-2}=
1+O\big(\displaystyle\sum_{j=1}^k\alpha^{p^*-1}_{j,\varepsilon}\big)+
O(\varepsilon^{m}),~\mbox{for all}~i=1,\cdots,k.
\end{equation}
Also, by Taylor expansion, we have
\begin{equation}\label{2.43}
(1+\alpha_{i,\varepsilon})^{p-2}=
1+(p-2)\alpha_{i,\varepsilon}+o(\alpha_{i,\varepsilon}),~\mbox{for all}~i=1,\cdots,k.
\end{equation}
Then \eqref{2.42} and \eqref{2.43} deduce
\begin{equation}\label{2.44}
\alpha_{i,\varepsilon}=O\big(\displaystyle\sum_{j=1}^k\alpha^{p^*-1}_{j,\varepsilon}\big)+
O(\varepsilon^{m}),~\mbox{for all}~i=1,\cdots,k.
\end{equation}
Using the fact $p^*>2$ and summing \eqref{2.44} from $i=1$ to $k$, we obtain \eqref{2.27}.
\end{proof}
\begin{Prop}\label{prop2.5}
Let $w_{\varepsilon}$ be as in \eqref{2.3}, suppose that \textup{($V_1$)} and \textup{($V_2$)} are satisfied, then we have
\begin{equation}\label{2.45}
\|w_{\varepsilon}\|_{\varepsilon}=O(\varepsilon^{m+\frac{N}{2}}).
\end{equation}
\end{Prop}
\begin{proof}
 It is obvious that \eqref{2.26} and \eqref{2.27} imply \eqref{2.45}.
\end{proof}

Furthermore, in next section, we need the following precise estimates about  $|x_{j,\varepsilon}-a_j|$.
\begin{Prop}\label{prop2.6}
Let $u_{\varepsilon}(x)$ be the solution of \eqref{1.1} with the form \eqref{2.1}, suppose that \textup{($V_1$)} and \textup{($V_2$)} are satisfied, then we have
\begin{equation*}
|x_{j,\varepsilon}-a_j|=o(\varepsilon), ~ j=1,\cdots,k.
\end{equation*}
\end{Prop}
\begin{proof}
From Lemma \ref{lem2.3} and Proposition \ref{prop2.4}, we have
\begin{equation}\label{2.46}
|x_{j,\varepsilon}-a_j|=O(\varepsilon).
\end{equation}
Then from \eqref{2.19} and \eqref{2.46}, we obtain
\begin{equation}\label{2.47}
\begin{aligned}
 \int_{B_{\frac{d}{\varepsilon}}(0)}\varepsilon^N|\varepsilon y+x_{j,\varepsilon}-a_j|^mU_{a_j}^2(y)\mathrm{d}y=O(\varepsilon^{N+m}).
\end{aligned}
\end{equation}
So \eqref{2.17}, \eqref{2.18} and \eqref{2.47} imply
\begin{equation}\label{2.48}
\begin{split}
\int_{B_{\frac{d}{\varepsilon}}(0)}|y+\frac{x_{j,\varepsilon}-a_j}{\varepsilon }|^{m-2}\cdot( y_i+\frac{x_{j,\varepsilon,i}-a_{j,i}}{\varepsilon})U_{a_j}^2(y)\mathrm{d}y
=O(\varepsilon)+O(\varepsilon^{m-1}),
\end{split}
\end{equation}
where $y_i$, $x_{j,\varepsilon,i}$, $a_{j,i}$ are the $i$-th components of $y$, $x_{j,\varepsilon}$, $a_{j}$ for $i=1,\cdots,N$.

By choosing a subsequence, we can suppose that $\frac{x_{j,\varepsilon}-a_j}{\varepsilon}\rightarrow x_0$. Then letting  $\varepsilon \rightarrow 0$ in \eqref{2.48}, we have
\begin{equation*}
\int_{\R^N}|y+x_0|^{m-2}\cdot(y_i
+x_{0,i})U_{a_j}^2(y)\mathrm{d}y=0,
\end{equation*}
where $x_{0,i}$ is the $i$-th component of $x_{0}$ for $i=1,\cdots,N$.

By the strictly decreasing of $U_{a_j}(x)$, we get $x_0=0$. That is $|x_{j,\varepsilon}-a_j|=o(\varepsilon)$.
\end{proof}
\section{Proof of the Main Theorem}
\setcounter{equation}{0}
Suppose that $u^{(1)}_{\varepsilon}(x)$, $u^{(2)}_{\varepsilon}(x)$ are two different positive solutions concentrating at $\{a_1,\cdots,a_k\}$, and
\begin{equation}\label{3.1}
\xi_{\varepsilon}(x)=\frac{u_{\varepsilon}^{(1)}(x)-u_{\varepsilon}^{(2)}(x)}
{\|u_{\varepsilon}^{(1)}-u_{\varepsilon}^{(2)}\|_{L^{\infty}(\R^N)}}.
\end{equation}
Then $\xi_{\varepsilon}(x)$ satisfies $\|\xi_{\varepsilon}\|_{L^{\infty}(\R^N)}=1$ and
\begin{equation}\label{3.2}
-\varepsilon^2\Delta \xi_{\varepsilon}(x)+V(x)\xi_{\varepsilon}(x)=C_{\varepsilon}(x)\xi_{\varepsilon}(x),
\end{equation}
where
\begin{equation*}
C_{\varepsilon}(x)=(p-1)\int_{0}^1\big(tu_{\varepsilon}^{(1)}(x)+(1-t)u_{\varepsilon}^{(2)}(x)\big)^{p-2}
\mathrm{d}t.
\end{equation*}

\begin{Prop}\label{prop3.1}
For $\xi_{\varepsilon}(x)$ defined by \eqref{3.1}, we have
\begin{equation}\label{3.3}
\|\xi_{\varepsilon}\|_{\varepsilon}=O(\varepsilon^{\frac{N}{2}}).
\end{equation}
\end{Prop}
\begin{proof}
From \eqref{3.2} we have
\begin{equation}\label{3.4}
\|\xi_{\varepsilon}\|^2_{\varepsilon}=\int_{\R^N}C_{\varepsilon}(x)
\xi_{\varepsilon}^2(x)\mathrm{d}x.
\end{equation}
On the other hand,
\begin{equation}\label{3.5}
|C_{\varepsilon}(x)|\leq
C\big(\sum_{j=1}^kU_{a_j}^{p-2}\large(\frac{x-x_{j,\varepsilon}^{(1)}}{\varepsilon}\big)
+\sum_{j=1}^kU_{a_j}^{p-2}\big(\frac{x-x_{j,\varepsilon}^{(2)}}{\varepsilon}\big)
+|w_{\varepsilon}^{(1)}(x)|^{p-2}+|w_{\varepsilon}^{(2)}(x)|^{p-2}\large).
\end{equation}
Since $|\xi_{\varepsilon}(x)|\leq 1$ and \eqref{2.4}. For $l=1,2$, we know
\begin{equation}\label{3.6}
\int_{\R^N}U_{a_j}^{p-2}(\frac{x-x_{j,\varepsilon}^{(l)}}{\varepsilon})\xi_{\varepsilon}^2(x)\mathrm{d}x\leq C\varepsilon^{N},
\end{equation}
and
\begin{equation}\label{3.7}
\begin{aligned}
\int_{\R^N}\big|w_{\varepsilon}^{(l)}(x)\big|^{p-2}\xi_{\varepsilon}^2(x)\mathrm{d}x&\leq \big(\int_{\R^N}\big|w_{\varepsilon}^{(l)}(x)\big|^{2^*}\mathrm{d}x\big)^{\frac{p-2}{2^{*}}}
\cdot\big(\int_{\R^N}\big|\xi_{\varepsilon}(x)\big|^{\frac{2\cdot2^*}{2^*-(p-2)}}\mathrm{d}x\big)^{1-\frac{p-2}{2^*}}\\
&\leq C\varepsilon^{(p-2)(\frac{N}{2}-1)}\|\xi_{\varepsilon}\|_{\varepsilon}^{2-\frac{(N-2)(p-2)}{N}}.
\end{aligned}
\end{equation}
Thus \eqref{3.4}, \eqref{3.5}, \eqref{3.6} and \eqref{3.7} imply
\begin{equation*}
\|\xi_{\varepsilon}\|_{\varepsilon}^2\leq C\big(\varepsilon^{N}+\varepsilon^{(p-2)(\frac{N}{2}-1)}\|\xi_{\varepsilon}\|
^{2-\frac{(N-2)(p-2)}{N}}_{\varepsilon}\big).
\end{equation*}
This leads to \eqref{3.3}.
\end{proof}

\begin{Lem}\label{3-2}

Let $\xi_{\varepsilon,j}(x)=\xi_{\varepsilon}(\varepsilon x+x_{j,\varepsilon}^{(1)})$, then taking
a subsequence necessarily, it holds
$$\xi_{\varepsilon,j}(x)\rightarrow\sum_{i=1}^N b_{j,i}\psi_{i}(x)$$
uniformly in $C^1(B_R(0))$ for any $R>0$, where $b_{j,i}$, $i=1,\cdots,N$ are some constants  and $$\psi_{i}(x)=\frac{\partial U_{a_j}(x)}{\partial x_i},~i=1,\cdots,N.$$
\end{Lem}
\begin{proof}
In view of $|\xi_{\varepsilon,j}(x)|\leq 1$, we may assume that $\xi_{\varepsilon,j}(x)\rightarrow\xi_{j}(x)$ in $C_{loc}(\R^N)$. By direct calculations, we have
\begin{equation}\label{3.8}
-\Delta\xi_{\varepsilon,j}(x)=-\varepsilon^2\Delta\xi_{\varepsilon}(\varepsilon x+x_{j,\varepsilon}^{(1)})=-V(\varepsilon x+x_{j,\varepsilon}^{(1)})\xi_{\varepsilon,j}(x)+C_{\varepsilon}(\varepsilon x+x_{j,\varepsilon}^{(1)})\xi_{\varepsilon,j}(x).
\end{equation}
Now, we estimate $C_{\varepsilon}(\varepsilon x+x_{j,\varepsilon}^{(1)})$,
\begin{equation*}
U_{a_s}\big(\frac{x-x_{s,\varepsilon}^{(1)}}{\varepsilon}\big)-U_{a_s}\big(\frac{x-x_{s,\varepsilon}^{(2)}}
{\varepsilon}\big)=O\big(\frac{x_{s,\varepsilon}^{(1)}-x_{s,\varepsilon}^{(2)}}{\varepsilon}\nabla U_{a_s}(\frac{x-x_{s,\varepsilon}^{(1)}}{\varepsilon})\big)=o(1)\nabla U_{a_s}\big(\frac{x-x_{s,\varepsilon}^{(1)}}{\varepsilon}\big),
\end{equation*}
for $s=1,\cdots,k$. Then
\begin{equation}\label{3.9}
u_{\varepsilon}^{(1)}(x)-u_{\varepsilon}^{(2)}(x)=o(1)\sum_{s=1}^k\nabla U_{a_s}\big(\frac{x-x_{s,\varepsilon}^{(1)}}{\varepsilon}\big)
+O\big(|w_{\varepsilon}^{(1)}(x)|+|w_{\varepsilon}^{(2)}(x)|\big).
\end{equation}
So, from \eqref{A.1}, for any $\gamma >0$ and $x\in B_{d}(x_{j,\varepsilon}^{(1)})$, we have
\begin{equation*}
C_{\varepsilon}(x)=(p-1)U_{a_j}^{p-2}\big(\frac{x-x_{j,\varepsilon}^{(1)}}{\varepsilon}\big)+\big(o(1)\nabla U_{a_j}\big(\frac{x-x_{j,\varepsilon}^{(1)}}{\varepsilon}\big)+O\big(|w_{\varepsilon}^{(1)}(x)|
+|w_{\varepsilon}^{(2)}(x)|
\big)\big)^{p-2}+o(\varepsilon^{\gamma}).
\end{equation*}
Then we have
\begin{equation*}
C_{\varepsilon}(\varepsilon x+x_{j,\varepsilon}^{(1)})=(p-1)U_{a_j}^{p-2}(x)+O\big(|w_{\varepsilon}^{(1)}(\varepsilon x+x_{j,\varepsilon}^{(1)})|+|w_{\varepsilon}^{(2)}(\varepsilon x+x_{j,\varepsilon}^{(1)})|\big)^{p-2}+o(1), ~x\in B_{\frac{d}{\varepsilon}}(0).
\end{equation*}
Next, for any given $\Phi(x)\in C_{0}^{\infty}(\R^N)$,
\begin{equation*}
\begin{aligned}
\displaystyle \int_{\R^N}&\big(-\Delta\xi_{\varepsilon,j}(x)+V(\varepsilon x+x_{j,\varepsilon}^{(1)})\xi_{\varepsilon,j}(x)-(p-1)U_{a_j}^{p-2}(x)\xi_{\varepsilon,j}
(x)\big)\Phi(x)\mathrm{d}x\\
&=O\big(\int_{\R^
N}\big(|w_{\varepsilon}^{(1)}(\varepsilon x+x_{j,\varepsilon}^{(1)})|^{p-2}+|w_{\varepsilon}^{(2)}(\varepsilon x+x_{j,\varepsilon}^{(1)})|^{p-2}\big)|\Phi(x)|\mathrm{d}x\big)+o(1).
\end{aligned}
\end{equation*}
Also, for $l=1,2$, we know
\begin{equation}\label{3.10}
\begin{split}
\displaystyle \int_{\R^N}|w_{\varepsilon}^{(l)}(\varepsilon x+x_{j,\varepsilon}^{(1)})|^{p-2}|\Phi(x)|\mathrm{d}x &\leq C\big(\int_{\R^N}|w_{\varepsilon}^{(l)}(\varepsilon x+x_{j,\varepsilon}^{(1)})|^{2^*}\mathrm{d}x\big)^{\frac{p-2}{2^*}}
\\&
\leq C\big(\varepsilon^{-1}\|w_{\varepsilon}^{(l)}\|_{\varepsilon}\big)^{p-2}
\varepsilon^{-\frac{N(p-2)}{2^*}}\\&=
C\big(\varepsilon^{-\frac{N}{2}}\|w_{\varepsilon}^{(l)}\|_{\varepsilon}\big)^{p-2}.
\end{split}
\end{equation}
Then using \eqref{2.4} and \eqref{3.10}, we can obtain
\begin{equation}\label{3.11}
\int_{\R^N}\big(-\Delta\xi_{\varepsilon,j}(x)+V(\varepsilon x+x_{j,\varepsilon}^{(1)})\xi_{\varepsilon,j}(x)-(p-1)U_{a_j}^{p-2}(x)\xi_{\varepsilon,j}
(x)\big)\Phi(x)\mathrm{d}x=o(1).
\end{equation}
Letting $\varepsilon\rightarrow 0$ in \eqref{3.11} and using the elliptic regularity theory, we find that $\xi_{j}(x)$ satisfies
\begin{equation*}
-\Delta\xi_{j}(x)+V(a_j)\xi_{j}(x)=(p-1)U_{a_j}^{p-2}(x)\xi_{j}(x),~~\textrm{in~}\R^N,
\end{equation*}
which gives
\begin{equation*}
\xi_{j}(x)=\sum_{i=1}^Nb_{j,i}\psi_i(x).
\end{equation*}
\end{proof}
\begin{Lem}\label{3-3}
Let $b_{j,i}$ be as in Lemma \ref{3-2}, then
we have
$$b_{j,i}=0,~~\mbox{for all}~j=1,\cdots,k,~i=1,\cdots N.$$
\end{Lem}

\begin{proof}
From Proposition \ref{prop2.5}, Proposition \ref{prop3.1} and Lemma \ref{lem-A-4}, for some small $\delta$, we have
\begin{equation}\label{3.12}
\big(\int_{\partial B_{\delta}(x_{j,\varepsilon}^{(1)})}\big(\varepsilon^2|\nabla w_{\varepsilon}(x)|^2+V(x)w^2_{\varepsilon}(x)\big)\mathrm{d}\sigma\big)^{\frac{1}{2}}
=O(\varepsilon^{\frac{N}{2}+m}),
\end{equation}
and
\begin{equation}\label{3.13}
\big(\int_{\partial B_{\delta}(x_{j,\varepsilon}^{(1)})}\big(\varepsilon^2|\nabla w_{\varepsilon}(x)|^2+V(x)w^2_{\varepsilon}(x)\big)\mathrm{d}\sigma\big)^{\frac{1}{2}}
=O(\varepsilon^{\frac{N}{2}}).
\end{equation}
Since $u_{\varepsilon}^{(1)}(x)$, $u_{\varepsilon}^{(2)}(x)$ are the positive solutions of \eqref{1.1}, using Pohozaev identity \eqref{2.5}, we deduce
\begin{equation}\label{3.14}
\begin{aligned}
&\displaystyle \int_{B_{\delta}(x_{j,\varepsilon}^{(1)})}\frac{\partial V(x)}{\partial x_i}\big(u_{\varepsilon}^{(1)}(x)+u_{\varepsilon}^{(2)}(x)\big)\cdot\xi_{\varepsilon}(x)\mathrm{d}x\\
&=-2\varepsilon^2\int_{\partial B_{\delta}(x_{j,\varepsilon}^{(1)})}\big(\frac{\partial \xi_{\varepsilon}(x)}{\partial \nu}\frac{\partial u_{\varepsilon}^{(1)}(x)}{\partial x_i}+\frac{\partial \xi_{\varepsilon}(x)}{\partial x_i}\frac{\partial u_{\varepsilon}^{(2)}(x)}{\partial \nu}\big)\mathrm{d}\sigma-2\int_{\partial B_{\delta}(x_{j,\varepsilon}^{(1)})}A_{\varepsilon}(x)\xi_{\varepsilon}(x)\nu_{i}(x)\mathrm{d}\sigma\\
&
~~~~+\int_{\partial B_{\delta}(x_{j,\varepsilon}^{(1)})}\big[\varepsilon^2\langle\nabla\big(u_{\varepsilon}^{(1)}(x)
+u_{\varepsilon}^{(2)}(x)\big),\nabla \xi_{\varepsilon}(x)\rangle+V(x)\langle u_{\varepsilon}^{(1)}(x)+u_{\varepsilon}^{(2)}(x),\xi_{\varepsilon}(x)\rangle
\big]\nu_i(x)\mathrm{d}\sigma,
\end{aligned}
\end{equation}
where $$A_{\varepsilon}(x)=\int_{0}^1\big(tu_{\varepsilon}^{(1)}(x)+(1-t)
u_{\varepsilon}^{(2)}(x)\big)^{p-1}\mathrm{d}t.$$
Then from \eqref{3.12},
\eqref{3.13}, \eqref{A.3} and \eqref{A.4}, we get
\begin{equation}\label{3.15}
\textrm{RHS of}~\eqref{3.14}=O(\varepsilon^{N+m})-2\int_{\partial B_{\delta}(x_{j,\varepsilon}^{(1)})}A_{\varepsilon}(x)\xi_{\varepsilon}(x)\nu_i(x)\mathrm{d}\sigma.
\end{equation}
Also, using Proposition \ref{prop2.6} and \eqref{A.2}, we have, for any $\gamma>0$,
\begin{equation*}
|A_{\varepsilon}(x)|\leq o(\varepsilon^{\gamma})+\big(|w^{(1)}_{\varepsilon}(x)|+|w^{(2)}_{\varepsilon}(x)|\big)^{p-1},~ x \in \partial B_{\delta}(x_{j,\varepsilon}^{(1)}).
\end{equation*}
Then from \eqref{2.4}, \eqref{A.11} and Lemma \ref{lem-A-4}, taking $\gamma>0$ appropriately, we get
\begin{equation}\label{3.16}
\begin{aligned}
|\int_{\partial B_{\delta}(x_{j,\varepsilon}^{(1)})}A_{\varepsilon}(x)\xi_{\varepsilon}(x)\nu_i(x)\mathrm{d}\sigma|&\leq C\big(\int_{\R^N}(|w^{(1)}_{\varepsilon}(x)|+|w^{(2)}_{\varepsilon}(x)|)^{p-1}
|\xi_{\varepsilon}(x)|\mathrm{d}x+\varepsilon^{\gamma}\big)\\&
\leq
C\big((\|w^{(1)}_{\varepsilon}\|_{\varepsilon}+\|w^{(2)}_{\varepsilon}\|_{\varepsilon})\|\xi_{\varepsilon}\|_{\varepsilon}+\varepsilon^{\gamma}\big)
\leq C\varepsilon^{N+m}.
\end{aligned}
\end{equation}
From \eqref{3.15} and \eqref{3.16},   we know
\begin{equation}\label{3.17}
\textrm{RHS of \eqref{3.14}}=O(\varepsilon^{N+m}).
\end{equation}
On the other hand,
\begin{equation}\label{3.18}
\begin{split}
\int_{B_{\delta}(x_{j,\varepsilon}^{(1)})}&\frac{\partial V(x)}{\partial x_i}\big(u_{\varepsilon}^{(1)}(x)+u_{\varepsilon}^{(2)}(x)\big)\xi_{\varepsilon}(x)\mathrm{d}x
\\&
=mb_i\int_{B_{\delta}(x_{j,\varepsilon}^{(1)})}|x-a_j|^{m-2}(x_i-a_{j,i})
\big(u_{\varepsilon}^{(1)}(x)+u_{\varepsilon}^{(2)}(x)\big)\xi_{\varepsilon}(x)\mathrm{d}x\\&
~~~~+O\big(\int_{B_{\delta}(x_{j,\varepsilon}^{(1)})}
|x-a_j|^{m}\big(u_{\varepsilon}^{(1)}(x)+u_{\varepsilon}^{(2)}(x)\big)\xi_{\varepsilon}(x)\mathrm{d}x\big).
\end{split}
\end{equation}
From \eqref{3.9}, we know
\begin{equation}\label{3.19}
\begin{split}
u_{\varepsilon}^{(1)}(x)+u_{\varepsilon}^{(2)}(x)&=2\sum_{s=1}^k U_{a_s}\big(\frac{x-x_{s,\varepsilon}^{(1)}}{\varepsilon}\big)+o(1)\sum_{s=1}^k\nabla U_{a_s}\big(\frac{x-x_{s,\varepsilon}^{(1)}}{\varepsilon}\big)\\&
~~~~
+O\big(|w_{\varepsilon}^{(1)}(x)|+|w_{\varepsilon}^{(2)}(x)|\big).
\end{split}
\end{equation}
Also, from \eqref{3.19} and \eqref{A.1}, we can get
\begin{equation}\label{3.20}
\begin{split}
\int_{B_{\delta}(x_{j,\varepsilon}^{(1)})}&|x-a_j|^{m-2}(x_i-a_{j,i})\big(u_{\varepsilon}^{(1)}(x)+
u_{\varepsilon}^{(2)}(x)\big)\xi_{\varepsilon}(x)\mathrm{d}x\\&
=2\int_{B_{\delta}(x_{j,\varepsilon}^{(1)})}|x-a_j|^{m-2}(x_i-a_{j,i})
U_{a_j}\big(\frac{x-x_{j,\varepsilon}^{(1)}}{\varepsilon}\big)\xi_{\varepsilon}(x)\mathrm{d}x
\\
&
~~~~+o(1)\int_{B_{\delta}(x_{j,\varepsilon}^{(1)})}|x-a_j|^{m-2}(x_i-a_{j,i})
\nabla U_{a_j}\big(\frac{x-x_{j,\varepsilon}^{(1)}}{\varepsilon}\big)\xi_{\varepsilon}(x)\mathrm{d}x
\\&
~~~~+O\big(\int_{B_{\delta}(x_{j,\varepsilon}^{(1)})}
(|w_{\varepsilon}^{(1)}(x)|+|w_{\varepsilon}^{(2)}(x)|)\xi_{\varepsilon}(x)\mathrm{d}x\big)
+O(\varepsilon^\gamma).
\end{split}
\end{equation}
Next, since $\psi_j(x)$ is an odd function with respect to $x_j$ and an even function with respect to $x_i$ for $i\neq j$, then using Proposition \ref{prop2.6} and
 Lemma \ref{3-2}, we have
\begin{equation}\label{3.21}
\begin{split}
&\int_{B_{\delta}(x_{j,\varepsilon}^{(1)})}|x-a_j|^{m-2}(x_i-a_{j,i})
U_{a_j}\big(\frac{x-x_{j,\varepsilon}^{(1)}}{\varepsilon}\big)\xi_{\varepsilon}(x)\mathrm{d}x\\&
=\int_{B_{\delta}(x_{j,\varepsilon}^{(1)})}|x-a_j|^{m-2}(x_i-a_{j,i})
U_{a_j}\big(\frac{x-x_{j,\varepsilon}^{(1)}}{\varepsilon}\big)\xi_{\varepsilon}(x)\mathrm{d}x\\&
=\varepsilon^{m+N-1}\big(\int_{B_{\frac{\delta}{\varepsilon}}(0)}|x+
\frac{x_{j,\varepsilon}^{(1)}-a_j}{\varepsilon}|^{m-2}
(x_i+\frac{x_{j,\varepsilon,i}^{(1)}-a_{j,i}}{\varepsilon})U_{a_j}(x)
\big(\sum_{l=1}^Nb_{j,l}\psi_l(x)\big)\mathrm{d}x+o(1)\big)\\&
=b_{j,i}\varepsilon^{m+N-1}\int_{\R^N}|x|^{m-2}x_iU_{a_j}(x)\psi_i(x)
\mathrm{d}x+o\big(\varepsilon^{m+N-1}\big).
\end{split}
\end{equation}
Also, \eqref{A.11} implies
\begin{equation}\label{3.22}
\int_{B_{\delta}(x_{j,\varepsilon}^{(1)})}
(|w_{\varepsilon}^{(1)}(x)|+|w_{\varepsilon}^{(2)}(x)|)\xi_{\varepsilon}(x)\mathrm{d}x=
O\big((\|w_{\varepsilon}^{(1)}\|_{\varepsilon}+\|w_{\varepsilon}^{(2)}\|_{\varepsilon})\|\xi_{\varepsilon}\|_{\varepsilon}\big)
=O(\varepsilon^{m+N}).
\end{equation}
Similar to \eqref{3.21} and \eqref{3.22}, we can deduce
\begin{equation}\label{3.23}
\begin{split}
\int_{B_{\delta}(x_{j,\varepsilon}^{(1)})}|x-a_j|^{m-2}(x_i-a_{j,i})
\nabla U_{a_j}\big(\frac{x-x_{j,\varepsilon}^{(1)}}{\varepsilon}\big)\xi_{\varepsilon}(x)\mathrm{d}x=
O\big(\varepsilon^{m+N-1}\big),
\end{split}
\end{equation}
and
\begin{equation}\label{3.24}
\int_{B_{\delta}(x_{j,\varepsilon}^{(1)})}
|x-a_j|^{m}\big(u_{\varepsilon}^{(1)}(x)+u_{\varepsilon}^{(2)}(x)\big)\xi_{\varepsilon}(x)\mathrm{d}x=O(\varepsilon^{m+N}).
\end{equation}
Then, by choosing $\gamma>0$ appropriately, from \eqref{3.18}, \eqref{3.20}, \eqref{3.21}, \eqref{3.22}, \eqref{3.23} and  \eqref{3.24}, we get
\begin{equation}\label{3.25}
\begin{split}
\int_{B_{\delta}(x_{j,\varepsilon}^{(1)})}&\frac{\partial V(x)}{\partial x_i}\big(u_{\varepsilon}^{(1)}(x)+u_{\varepsilon}^{(2)}(x)\big)\xi_{\varepsilon}(x)\mathrm{d}x\\&
=
2mb_ib_{j,i}\varepsilon^{m+N-1}\int_{\R^N}|x|^{m-2}x_iU_{a_j}(x)\psi_i(x)
\mathrm{d}x+o\big(\varepsilon^{m+N-1}\big).
\end{split}
\end{equation}
So \eqref{3.14}, \eqref{3.17} and \eqref{3.25} imply
\begin{equation*}
2mb_ib_{j,i}\int_{\R^N}|x|^{m-2}x_iU_{a_j}(x)\psi_i(x)\mathrm{d}x=o(1).
\end{equation*}
This means $b_{j,i}=0$. Similarly, we can obtain
$b_{j,i}=0$, for all  $j=1,\cdots,k$, $i=1,\cdots,N$.
\end{proof}
\begin{Prop}\label{prop3.4}
For any fixed $R>0$, it holds
$$\xi_{\varepsilon}(x)=o(1),~ x\in \bigcup_{j=1}^kB_{R\varepsilon}(x_{j,\varepsilon}^{(1)}).$$
\end{Prop}
\begin{proof}
Lemma \ref{3-2} and Lemma \ref{3-3} show that for any fixed $R>0$,
\begin{equation*}
\xi_{\varepsilon,j}(x)=o(1)~\mbox{in}~ B_{R}(0),~j=1,\cdots,k.
\end{equation*}
Also, we know $\xi_{\varepsilon,j}(x)=\xi_{\varepsilon}(\varepsilon x+x_{j,\varepsilon}^{(1)})$, then $\xi_{\varepsilon}(x)=o(1),x\in B_{R\varepsilon}(x_{j,\varepsilon}^{(1)})$.
\end{proof}
\begin{Prop}\label{prop3.5}
For large $R>0$, we have
\begin{equation*}
\xi_{\varepsilon}(x)=o(1),~ x\in \R^N\backslash\bigcup_{j=1}^k B_{R\varepsilon}(x_{j,\varepsilon}^{(1)}).
\end{equation*}
\end{Prop}

\begin{proof}
First, we have
\begin{equation*}
-\varepsilon^2\Delta\xi_{\varepsilon}(x)+V(x)\xi_{\varepsilon}(x)=
C_{\varepsilon}(x)\xi_{\varepsilon}(x),
\end{equation*}
where
$$C_{\varepsilon}(x)=(p-1)\int_{0}^1\big(tu_{\varepsilon}^{(1)}(x)+(1-t)
u_{\varepsilon}^{(2)}(x)\big)^{p-2}\mathrm{d}t.$$
Next, for $x\in \R^N \backslash\bigcup_{j=1}^kB_{R\varepsilon}(x_{j,\varepsilon}^{(1)})$, we know
\begin{equation}\label{3.26}
|C_{\varepsilon}(x)|\leq C\big(|w_{\varepsilon}^{(1)}(x)|^{p-2}+
|w_{\varepsilon}^{(2)}(x)|^{p-2}\big)+o_{R}(1)+o_{\varepsilon}(1),~\textrm{as}~R\rightarrow\infty,
\varepsilon\rightarrow 0.
\end{equation}
Now, we estimate $w_{\varepsilon}^{(l)}(x)$ in $\R^N \backslash\bigcup_{j=1}^kB_{R\varepsilon}(x_{j,\varepsilon}^{(1)})$, $l=1,2$.

Setting $\tilde {w}_{\varepsilon}^{(l)}(x)=w_{\varepsilon}^{(l)}(\varepsilon x)$, then from \eqref{3.2}, we obtain
\begin{equation}\label{3.27}
-\Delta \tilde {w}_{\varepsilon}^{(l)}(x)+V(\varepsilon x)\tilde {w}_{\varepsilon}^{(l)}(x)=\tilde{N}(\tilde {w}_{\varepsilon}^{(l)}(x))+\tilde{l}_{\varepsilon}(\varepsilon x),~~ x\in \R^N\backslash\bigcup_{j=1}^kB_{R}(x_{j,\varepsilon}^{(1)}),
\end{equation}
where
\begin{equation*}
\begin{cases}
\Tilde{N}(\tilde {w}_{\varepsilon}^{(l)}(x))=\big(\displaystyle\sum_{j=1}^kU_{a_j}(\frac{\varepsilon x-x_{j,\varepsilon}^{(l)}}{\varepsilon})+\tilde {w}_{\varepsilon}^{(l)}(x)\big)^{p-1}-\displaystyle\sum_{j=1}^kU^{p-1}_{a_j}(\frac{\varepsilon x-x_{j,\varepsilon}^{(l)}}{\varepsilon}),\\
\tilde{l}_{\varepsilon}(\varepsilon x)=\displaystyle\sum_{j=1}^k\big(V(\varepsilon a_j)-V(\varepsilon x)\big)U_{a_j}\big(\frac{\varepsilon x-x_{j,\varepsilon}^{(l)}}{\varepsilon}\big).
\end{cases}
\end{equation*}
From \eqref{2.4} we deduce $\|\tilde {w}_{\varepsilon}^{(l)}\|_{\varepsilon}=o(1)$, and by the exponential decay of $U_{a_{j}}(\frac{\varepsilon x-x_{j,\varepsilon}^{(l)}}{\varepsilon})$ in $W$ as $R\rightarrow\infty$, then
\begin{equation*}
\|\tilde {w}_{\varepsilon}^{(l)}(x)\|_{H^1(W)}=o_{\varepsilon}(1)+o_{R}(1),~\textrm{as }~\varepsilon\rightarrow 0,~R\rightarrow \infty,
\end{equation*}
where $W=\R^N\backslash\bigcup_{j=1}^kB_{R}(x_{j,\varepsilon}^{(l)})$.
So,
\begin{equation}\label{3.28}
\|\Tilde{N}(\tilde {w}_{\varepsilon}^{(l)}(x))\|_{L^{\frac{2^*}{p-1}}(W)}=o_{\varepsilon}(1),
~\mbox{and}~\|\tilde{l}_{\varepsilon}(\varepsilon x)\|_{L^{q}(W)}=o_{R}(1),~\forall q>1.
\end{equation}
Combining \eqref{3.27}, \eqref{3.28} and the $L^p$ estimates, we have
\begin{equation*}
\|\tilde {w}_{\varepsilon}^{(l)}(x)\|_{W^{2,\frac{2^*}{p-1}}(W)}=o_{\varepsilon}(1)+o_{R}(1).
\end{equation*}
Then, using the Sobolev embedding theorems and $L^p$ estimates for finite steps, we obtain
\begin{equation*}
\|\tilde {w}_{\varepsilon}^{(l)}(x)\|_{W^{2,q}(W)}=o_{\varepsilon}(1)+o_{R}(1),~\textrm{for some $q\in(\frac{N}{2},N)$}.
\end{equation*}
Next, using Sobolev embedding theorems again, we have
\begin{equation*}
\|\tilde {w}_{\varepsilon}^{(l)}(x)\|_{L^{\infty}(W)}\leq C\|\tilde {w}_{\varepsilon}^{(l)}(x)\|_{C^{0,2-\frac{N}{q}}(W)}=o_{\varepsilon}(1)+o_{R}(1).
\end{equation*}
This means
\begin{equation}\label{3.29}
\| {w}_{\varepsilon}^{(l)}(x)\|_{L^{\infty}\big(\R^N\backslash\bigcup_{j=1}^kB_{R\varepsilon}
(x_{j,\varepsilon}^{(1)})\big)}=o_{\varepsilon}(1)+o_{R}(1).
\end{equation}
Then \eqref{3.26} and \eqref{3.29} show that for large $R$ and small $\varepsilon$,
\begin{equation*}
|C_{\varepsilon}(x)|\leq \inf_{x\in\R^N}V(x).
\end{equation*}
Thus for large $R$, we have
\begin{equation*}
  \begin{cases}
    -\varepsilon^2\Delta \xi_{\varepsilon}+\big(V(x)-C_{\varepsilon}(x)\big)\xi_{\varepsilon}=0,
    &\text{$x\in\R^N\backslash\bigcup_{j=1}^kB_{R\varepsilon}(x_{j,\varepsilon}^{(1)})$},\\
    \xi_{\varepsilon}(x)=o(1)~(\textrm{as~$\varepsilon\rightarrow 0$}), &\text{$x
    \in \partial \big(\bigcup_{j=1}^kB_{R\varepsilon}(x_{j,\varepsilon}^{(1)})$}\big),\\
    \xi_{\varepsilon}(x)\rightarrow 0,&\text{as $|x|\rightarrow 0$},
  \end{cases}
\end{equation*}
and
$$V(x)-C_{\varepsilon}(x)\geq 0,~x\in\R^N\backslash\bigcup_{j=1}^kB_{R\varepsilon}(x_{j,\varepsilon}^{(1)}).$$
By the maximum principle, we obtain
\begin{equation*}
\xi_{\varepsilon}(x)=o(1),~ x\in\R^N\backslash\bigcup_{j=1}^kB_{R\varepsilon}(x_{j,\varepsilon}^{(1)}).
\end{equation*}
\end{proof}
\begin{Rem}
Since the nonlinear term of problem \eqref{1.1} is  subcritical, we can not obtain the pointwise estimate of the error term $w_{\varepsilon}(x)$ by the similar methods in \cite{Deng}. In our paper, we use the estimate of the norm $\|w_{\varepsilon}\|_{\varepsilon}$ to prove Proposition \ref{prop3.4}. On the other hand,  in  Proposition \ref{prop3.5} we mainly use the
technique of maximum principle.
\end{Rem}
\begin{proof}[\textbf{Proof of Theorem \ref{th1.1}:}]
Suppose that $u^{(1)}_{\varepsilon}(x)$, $u^{(2)}_{\varepsilon}(x)$ are two different positive solutions concentrating at  $k$ different points  $\{a_1,\cdots,a_k\}$. From Proposition \ref{prop3.4} and Proposition \ref{prop3.5}, for small $\varepsilon$, we have
\begin{equation*}
\xi_{\varepsilon}(x)=o(1),~ x\in\R^N,
\end{equation*}
which is in contradiction with $\|\xi_{\varepsilon}\|_{L^{\infty}(\R^N)}=1$. So, $u^{(1)}_{\varepsilon}(x)\equiv u^{(2)}_{\varepsilon}(x)$ for small $\varepsilon$. \eqref{1.7}
follows from Proposition \ref{prop2.1}, Proposition \ref{prop2.5} and Proposition \ref{prop2.6}.
\end{proof}




\section*{Appendix}

In this appendix, we give various estimates and results which have been used repeatedly
in previous sections.

\appendix
\renewcommand{\theequation}{A.\arabic{equation}}
\section{Some Estimates}
\setcounter{equation}{0}

First, from the exponential decay of $U_{a_j}(x)$ for $j=1,\cdots,k$, we have
\begin{Lem}
There exists a small constant $d_1$, such that, for any $\gamma>0$, we have
 \begin{equation}\label{A.1}
  U_{a_j}(\frac{x-x_{j,\varepsilon}}{\varepsilon})=O(\varepsilon^{\gamma}), ~\mbox{for}~ x\in B_{d}(x_{i,\varepsilon}),~j\neq i~\mbox{and}~0<d<d_1,
 \end{equation}
and
 \begin{equation}\label{A.2}
   U_{a_j}(\frac{x-x_{j,\varepsilon}}{\varepsilon})=O(\varepsilon^{\gamma}), ~\mbox{for}~
x\in \partial B_{d}(x_{i,\varepsilon}),~i=1,\cdots,k ~\mbox{and}~0<d<d_1.
 \end{equation}
\end{Lem}
\begin{Lem}
For any $\gamma>0$, it holds
 \begin{equation}\label{A.3}
  \int_{\R^{N}}U_{a_i}^{q_1}\big(\frac{x-x_{i,\varepsilon}}{\varepsilon}\big)
   U^{q_2}_{a_j}\big(\frac{x-x_{j,\varepsilon}}{\varepsilon}\big)
   \mathrm{d}x=O(\varepsilon^{\gamma}),
 \end{equation}
and
 \begin{equation}\label{A.4}
  \int_{\R^{N}}\varepsilon^2\nabla U_{a_i}\big(\frac{x-x_{i,\varepsilon}}{\varepsilon}\big)
   \nabla U_{a_j}\big(\frac{x-x_{j,\varepsilon}}{\varepsilon}\big)
   \mathrm{d}x=O(\varepsilon^{\gamma}),
 \end{equation}
where $i,j=1,\cdots,k$, $i\neq j$ and $q_1, q_2>0$.
\end{Lem}
\begin{proof}
Taking a small constant $d$ and using \eqref{A.1}, for any $\gamma>0$, we have
 \begin{equation*}
 \int_{B_{d}(x_{i,\varepsilon})}U^{q_1}_{a_i}
    (\frac{x-x_{i,\varepsilon}}{\varepsilon})U^{q_2}_{a_j}
    (\frac{x-x_{j,\varepsilon}}{\varepsilon})\mathrm{d}x\leq C\varepsilon^{\gamma}\int_{B_{d}(x_{i,\varepsilon})}U^{q_1}_{a_i}
    (\frac{x-x_{i,\varepsilon}}{\varepsilon})\mathrm{d}x\leq  C\varepsilon^{\gamma}.
 \end{equation*}
Similarly,
\begin{equation*}
\int_{B_{d}(x_{j,\varepsilon})}U^{q_1}_{a_i}
    (\frac{x-x_{i,\varepsilon}}{\varepsilon})U^{q_2}_{a_j}
    (\frac{x-x_{j,\varepsilon}}{\varepsilon})\mathrm{d}x
    \leq  C\varepsilon^{\gamma},
 \end{equation*}
 and
 \begin{equation*}
  \int_{\R^{N}\backslash\big(B_{d}(x_{i,\varepsilon})\bigcup B_{d}(x_{j,\varepsilon})\big)}U^{q_1}_{a_i}
    (\frac{x-x_{i,\varepsilon}}{\varepsilon})U^{q_2}_{a_j}
    (\frac{x-x_{j,\varepsilon}}{\varepsilon})\mathrm{d}x
    \leq  C\varepsilon^{\gamma}.
 \end{equation*}
The above inequalities imply \eqref{A.3}.  Also combining \eqref{1.4}, \eqref{A.2} and the proof of \eqref{A.3}, we know \eqref{A.4}.
\end{proof}

\begin{Lem}\label{lem-A-1}
 Suppose that $V(x)$ satisfies \eqref{1.5}, then  we have
 \begin{equation}\label{A.5}
  \int_{\R^N}\big(V(a_j)-V(x)\big)U_{a_j}\big(\frac{x-x_{j,\varepsilon}}{\varepsilon}\big)
u(x)\mathrm{d}x=O\big(\varepsilon^{\frac{N}{2}+m}+\varepsilon^
   {\frac{N}{2}}\max_{j=1,\cdots,k.}|x_{j,\varepsilon}-a_j|^m\big)\|u\|_{\varepsilon},
 \end{equation}
and
\begin{equation}\label{A.6}
\begin{split}
\int_{\R^N}\frac{\partial V(x)}{\partial x_i}U_{a_j}\big(\frac{x-x_{j,\varepsilon}}{\varepsilon}\big)u(x)\mathrm{d}x=
O\big(\varepsilon^{\frac{N}{2}+m-1}+\varepsilon^
   {\frac{N}{2}}\max_{j=1,\cdots,k.}|x_{j,\varepsilon}-a_j|^{m-1}\big)\|u\|_{\varepsilon},
\end{split}
\end{equation}
 where $u(x)\in H_{\varepsilon}$ and $j=1,\cdots,k$.
\end{Lem}
\begin{proof}
First, from \eqref{1.5} and H\"older's inequality, for a small constant $d$, we have
 \begin{equation}\label{A.7}
   \begin{split}
    \big|\int_{B_{d}(x_{j,\varepsilon})}&\big(V(a_j)-V(x)\big)U_{a_j}\big(\frac{x-x_{j,\varepsilon}}
    {\varepsilon}\big)u(x)\mathrm{d}x\big|\\&
    \leq C\int_{B_{d}(x_{j,\varepsilon})}|x-a_{j}|
    ^mU_{a_j}\big(\frac{x-x_{j,\varepsilon}}{\varepsilon}\big)|u(x)|\mathrm{d}x
    \\&\leq C\big(\int_{B_{d}(x_{j,\varepsilon})}|x-a_j|^{2m}U_{a_j}^2\big(\frac{x-x_{j,\varepsilon}}
    {\varepsilon}\big)\mathrm{d}x\big)^{\frac{1}{2}}\big(\int_{B_{d}(x_{j,\varepsilon})}u^2(x)
    \mathrm{d}x\big)^{\frac{1}{2}}\\
    &\leq C\varepsilon^{\frac{N}{2}}\big(\int_{B_{\frac{d}{\varepsilon}}(0)}|\varepsilon y+(x_{j,\varepsilon}-a_j)|^{2m}U_{a_j}^2(y)\mathrm{d}y\big)^{\frac{1}{2}}\|u\|_{\varepsilon}\\
    &\leq C\varepsilon^{\frac{N}{2}}\big(\varepsilon^m+|x_{j,\varepsilon}-a_j|^m\big)\|u
    \|_{\varepsilon}.
   \end{split}
 \end{equation}
Also, by the exponential decay of $U_{a_j}\big(\frac{x-x_{j,\varepsilon}}{\varepsilon}\big)$ in
$\R^N\backslash B_{d}(x_{j,\varepsilon})$, we can deduce that, for any $\gamma>0$,
 \begin{equation}\label{A.8}
   \big|\int_{\R^N\backslash B_{d}(x_{j,\varepsilon})}\big(V(a_j)-V(x)\big)U_{a_j}
   \big(\frac{x-x_{j,\varepsilon}}{\varepsilon}\big)u(x)
   \mathrm{d}x\big|\leq C\varepsilon^{\gamma}\|u\|_{\varepsilon}.
 \end{equation}
Then, taking suitable $\gamma>0$, from  \eqref{A.7} and \eqref{A.8}, we get \eqref{A.5}.

Next, similar to \eqref{A.7} and \eqref{A.8}, for any $\gamma>0$, we have
\begin{equation}\label{A.9}
\begin{split}
\big|\int_{B_{d}(x_{j,\varepsilon})}\frac{\partial V(x)}{\partial x_i}U_{a_j}\big(\frac{x-x_{j,\varepsilon}}{\varepsilon}\big)u(x)\mathrm{d}x\big|&
\leq C\int_{B_{d}(x_{j,\varepsilon})}|x-a_{j}|
    ^{m-1}U_{a_j}\big(\frac{x-x_{j,\varepsilon}}{\varepsilon}\big)|u(x)|\mathrm{d}x
\\&\leq C\varepsilon^{\frac{N}{2}}\big(\varepsilon^{m-1}+|x_{j,\varepsilon}-a_j|^{m-1}\big)\|u
    \|_{\varepsilon},
\end{split}
\end{equation}
and
 \begin{equation}\label{A.10}
   \big|\int_{\R^N\backslash B_{d}(x_{j,\varepsilon})}
\frac{\partial V(x)}{\partial x_i}U_{a_j}\big(\frac{x-x_{j,\varepsilon}}{\varepsilon}\big)u(x)\mathrm{d}x\big|\leq C\varepsilon^{\gamma}\|u\|_{\varepsilon}.
 \end{equation}
Then, taking suitable $\gamma>0$, \eqref{A.9} and \eqref{A.10} imply \eqref{A.6}.
\end{proof}

\begin{Lem}\label{lem-A--5}
Suppose that $u(x),v(x)\in H_{\varepsilon}$ and $\|u\|_{\varepsilon}=o(\varepsilon^{\frac{N}{2}})$, then it holds
\begin{equation}\label{A.11}
\begin{aligned}
\int_{\R^N}u^{p-1}(x)v(x)\mathrm{d}x=
o\big( \|u\|_{\varepsilon}
\|v\|_{\varepsilon}),~\mbox{for any}~2<p<\frac{2N}{N-2}.
\end{aligned}
\end{equation}
\end{Lem}
\begin{proof}
First, for any $s\in (2,\frac{2N}{N-2})$, using H\"older's inequality and the fact that $\|u\|_{\varepsilon}=o(\varepsilon^{\frac{N}{2}})$, we can deduce
\begin{equation}\label{A.12}
\begin{split}
\int_{\R^N}|u(x)|^s\mathrm{d}x&\leq C\big(\int_{\R^N}|u(x)|^{2^*}\mathrm{d}x\big)^{\frac{(N-2)(s-2)}{4}}
\cdot\big(\int_{\R^N}|u(x)|^{2}\mathrm{d}x\big)^{\frac{2N-p(N-2)}{4}}\\&
\leq C\big(\varepsilon^{-1}\|u\|_{\varepsilon}\big)^{\frac{N(s-2)}{2}}\cdot \|u\|_{\varepsilon}^{\frac{2N-s(N-2)}{2}}\\&
\leq C\varepsilon^{-\frac{N(s-2)}{2}}\|u\|_{\varepsilon}^{s}
=o\big(\|u\|^2_{\varepsilon}\big).
\end{split}
\end{equation}
Next, by H\"older's inequality and  Sobolev inequality,  we see
\begin{equation}\label{A.13}
\begin{aligned}
\big|\int_{\R^N}u^{p-1}(x)v(x)\mathrm{d}x\big|
&\leq \begin{cases}
\big(\int_{\R^N}|u(x)|^{2(p-1)}\mathrm{d}x\big)^{\frac{1}{2}}
\|v\|_{\varepsilon},&2<p<\frac{2N-2}{N-2};
\\\big(\int_{\R^N}|u(x)|^{\frac{2Np}{N+2}}\mathrm{d}x\big)^{\frac{N+2}{2N}}\cdot
\big(\varepsilon^{-1}\|v\|_{\varepsilon}\big),&\frac{2N-2}{N-2}\leq p<\frac{2N}{N-2}.\end{cases}
\end{aligned}
\end{equation}
Then from $\|u\|_{\varepsilon}=o(\varepsilon^{\frac{N}{2}})$, \eqref{A.12} and \eqref{A.13}, we get \eqref{A.11}.
\end{proof}

\begin{Lem}\label{lem-A-4}
For any fixed number $l\in \N^+$, suppose that $\{u_i(x)\}^l_{i=1}$ satisfies
$$\int_{\R^N}|u_i(x)|\mathrm{d}x<+\infty, ~i=1,\cdots,l.$$
Then for any $x_0$, there exist a small constant $d$ and another constant $C$ such that
\begin{equation}\label{A.14}
\int_{\partial B_{d}(x_0)}|u_i(x)|\mathrm{d}\sigma \leq C \int_{\R^N}|u_i(x)|\mathrm{d}x, ~\mbox{for all}~i=1,\cdots,l.
\end{equation}
\end{Lem}
\begin{proof}
Let $M_i=\int_{\R^N}|u_i(x)|\mathrm{d}x$, for $i=1,\cdots,l$. Then for a fixed small $r_0>0$,
\begin{equation}\label{A.15}
\int_{ B_{r_0}(x_0)}\big(\sum^l_{i=1}|u_i(x)|\big)\mathrm{d}x \leq\sum^{l}_{i=1}M_i, ~\mbox{for all}~i=1,\cdots,l.
\end{equation}
On the other hand,
\begin{equation}\label{A.16}
\int_{ B_{r_0}(x_0)}\big(\sum^l_{i=1}|u_i(x)|\big)\mathrm{d}x \geq \int^{r_0}_{0}\int_{ \partial B_{r}(x_0)}\big(\sum^l_{i=1}|u_i(x)|\big)\mathrm{d}\sigma dr.
\end{equation}
Then \eqref{A.15} and \eqref{A.16} imply that there exists a constant $d<r_0$ such that
\begin{equation}\label{A.17}
\int_{ \partial B_{r}(x_0)}|u_i(x)|\mathrm{d}\sigma\leq \int_{ \partial B_{d}(x_0)}\big(\sum^l_{i=1}|u_i(x)|\big)\mathrm{d}\sigma\leq  \frac{\sum^{l}_{i=1}M_i}{r_0}, ~\mbox{for all}~i=1,\cdots,l.
\end{equation}
So taking $C=\displaystyle\max_{1\leq i\leq l}{\frac{\sum^{l}_{i=1}M_i}{r_0M_i}}$, we can obtain \eqref{A.14} from \eqref{A.17}.
\end{proof}
\section{Analysis of $w_{\varepsilon}$ and $\upsilon_{\varepsilon}$}

 \renewcommand{\theequation}{B.\arabic{equation}}
\setcounter{equation}{0}
By Proposition \ref{prop2.1}, let
$$u_{\varepsilon}(x)=\displaystyle\sum_{j=1}^k U_{a_j}(\frac{x-x_{j,\varepsilon}}{\varepsilon})+w_{\varepsilon}(x)=\sum_{j=1}^k(1+\alpha_{j,\varepsilon})U_{a_j}(\frac{x-x_{j,\varepsilon}}
{\varepsilon})+\upsilon_{\varepsilon}(x)$$
be  the solution of \eqref{1.1},  then we have
 \begin{equation}\label{B.1}
    -\varepsilon^2\Delta w_{\varepsilon}(x)+V(x)w_{\varepsilon}(x)-(p-1)\displaystyle\sum_{j=1}^kU_{a_j}^{p-2}
    (\frac{x-x_{j,\varepsilon}}{\varepsilon})w_{\varepsilon}(x)
    =N\big(w_{\varepsilon}(x)\big)+l_{\varepsilon}(x),
 \end{equation}
where
 \begin{equation*}
   \begin{cases}
      N\big(w_{\varepsilon}(x)\big)=\big(\displaystyle\sum_{j=1}^k U_{a_j}(\frac{x-x_{j,\varepsilon}}{\varepsilon})+w_{\varepsilon}(x)\big)^{p-1}-
      \displaystyle\sum_{j=1}^kU_{a_j}^{p-1}(\frac{x-x_{j,\varepsilon}}{\varepsilon})\\
      ~~~~~~~~~~~~~~~~~-(p-1)\displaystyle
      \sum_{j=1}^kU_{a_j}^{p-2}(\frac{x-x_{j,\varepsilon}}{\varepsilon})w_{\varepsilon}(x),\\
      l_{\varepsilon}(x)=\displaystyle\sum_{j=1}^k(V(a_j)-V(x))U_{a_j}
      (\frac{x-x_{j,\varepsilon}}{\varepsilon}).
    \end{cases}
 \end{equation*}

\begin{Lem}\label{lem-A.1}
There exists a constant $\rho>0$ independent of $\varepsilon$ such that
 \begin{equation}\label{B.2}
   \int_{\R^N}\big[\varepsilon^2|\nabla u(x)|^2+V(x)u^2(x)
   -(p-1)\displaystyle\sum_{j=1}^kU_{a_j}^{p-2}(\frac{x-x_{j,\varepsilon}}{\varepsilon})u^2(x)\big]\mathrm{d}x\geq \rho \|u\|^2_{\varepsilon},
 \end{equation}
for all $u(x)\in \bigcap^k_{j=1}E_{\varepsilon,a_j,x_{j,\varepsilon}}$ and $x_{j,\varepsilon}\in B_{\delta}(a_j)$ provided $\delta>0$ is sufficiently small.
\end{Lem}

\begin{proof}
See [\cite{Cao2}, Proposition C.1].
\end{proof}
\begin{Prop}\label{prop-A.2}
It holds
  \begin{equation}\label{B.3}
    \|w_{\varepsilon}\|^2_{\varepsilon}=O(\varepsilon^{N+2m})+O\big(\varepsilon^
    {N}\max_{j=1,\cdots,k}|x_{j,\varepsilon}-a_j|^{2m}\big)+O(\varepsilon^{N}
    \sum^{k}_{j=1}\alpha_{j,\varepsilon}^2),
  \end{equation}
and
\begin{equation}\label{B.4}
    \|\upsilon_{\varepsilon}\|^2_{\varepsilon}=O(\varepsilon^{N+2m})+O\big(\varepsilon^
    {N}\max_{j=1,\cdots,k}|x_{j,\varepsilon}-a_j|^{2m}\big)+O(\varepsilon^{N}\sum^{k}_{j=1}\alpha_{j,\varepsilon}^2).
  \end{equation}
\end{Prop}
\begin{proof}
First, from  Lemma \ref{lem-A.1}, we know
 \begin{equation}\label{B.5}
   \begin{split}
   \|\upsilon_{\varepsilon}\|_{\varepsilon}^2&\leq C\big(\int_{\R^N}\big[\varepsilon^2 |\nabla \upsilon_{\varepsilon}(x)|^2+V(x)\upsilon_{\varepsilon}^2(x)
   -(p-1)\displaystyle\sum_{j=1}^{k}U_{a_j}^{p-2}
    (\frac{x-x_{j,\varepsilon}}{\varepsilon})\upsilon_{\varepsilon}^2(x)\big]\mathrm{d}x\big)
    \\&=C\big(\int_{\R^N}\big[\varepsilon^2|\nabla w_{\varepsilon}(x)|^2+V(x)w_{\varepsilon}^2(x)
   -(p-1)\displaystyle\sum_{j=1}^{k}U_{a_j}^{p-2}
    (\frac{x-x_{j,\varepsilon}}{\varepsilon})w_{\varepsilon}^2(x)\big]\mathrm{d}x\big)\\&
    ~~~~+C\big(A_1+A_2\big),
   \end{split}
 \end{equation}
where
 \begin{equation*}
 \begin{split}
    A_1&=-\|\displaystyle\sum_{j=1}^k\alpha_{j,\varepsilon}U_{a_j}
    (\frac{x-x_{j,\varepsilon}}{\varepsilon})\|_{\varepsilon}+
    (p-1)\displaystyle\sum_{l=1}^{k}\int_{\R^N}U_{a_l}^{p-2}
    (\frac{x-x_{l,\varepsilon}}{\varepsilon})\big
    (\sum_{j=1}^k\alpha_{j,\varepsilon}U_{a_j}
    (\frac{x-x_{j,\varepsilon}}{\varepsilon})\big)^2\mathrm{d}x,
 \end{split}
 \end{equation*}
and
 \begin{equation*}
   \begin{split}
     A_2&=
     2(p-1)\displaystyle\sum_{j=1}^k\displaystyle\sum_{l=1}^k\alpha_{j,\varepsilon}
    \int_{\R^N}U_{a_l}^{p-2} (\frac{x-x_{l,\varepsilon}}{\varepsilon})U_{a_j}
     (\frac{x-x_{j,\varepsilon}}{\varepsilon})\upsilon_{\varepsilon}(x)
     \mathrm{d}x\\&
     ~~~~-2\displaystyle\sum_{j=1}^k\alpha_{j,\varepsilon}\big(U_{a_j}(\frac{x-x_{j,\varepsilon}}{\varepsilon}), \upsilon_{\varepsilon}(x)\big)_{\varepsilon}.
     \end{split}
 \end{equation*}
 From \eqref{A.3} and \eqref{A.4}, for any $\gamma>0$, we know
 \begin{equation}\label{B.6}
 \begin{split}
    A_1=O\big(\sum_{j=1}^k\alpha_{j,\varepsilon}^2\varepsilon^N\big)+O(\varepsilon^{\gamma}).
 \end{split}
 \end{equation}
 Then $\upsilon_\varepsilon(x)\in \bigcap^k_{j=1}E_{\varepsilon,a_j,x_{j,\varepsilon}}$, \eqref{1.3} and \eqref{A.3}
 imply that, for any $\gamma>0$,
 \begin{equation}\label{B.7}
 \begin{split}
    A_2&= 2(p-1)\displaystyle\sum_{j=1}^k\alpha_{j,\varepsilon}
    \int_{\R^N}U_{a_j}^{p-1} (\frac{x-x_{j,\varepsilon}}{\varepsilon})\upsilon_{\varepsilon}(x)
     \mathrm{d}x+O(\varepsilon^{\gamma})\\&
     = 2(p-1)\displaystyle\sum_{j=1}^k\alpha_{j,\varepsilon}\int_{\R^N}
     \big(V(a_j)-V(x)\big)U_{a_j}\big(\frac{x-x_{j,\varepsilon}}{\varepsilon}\big)
  \upsilon_{\varepsilon}(x)\mathrm{d}x+O(\varepsilon^{\gamma}).
 \end{split}
 \end{equation}
Then \eqref{A.5} and \eqref{B.7} imply
 \begin{equation}\label{B.8}
   A_2=O\big(\sum_{j=1}^k\alpha_{j,\varepsilon}(\varepsilon^{\frac{N}{2}+m}+\varepsilon^{\frac{N}{2}}
   \displaystyle\max_{j=1,\cdots,k}|x_{j,\varepsilon}-a_j|^{m})
   \big)\|\upsilon_{\varepsilon}\|_{\varepsilon}+O(\varepsilon^{\gamma}).
 \end{equation}
So, combining \eqref{B.5}, \eqref{B.6} and \eqref{B.8}, for any $\gamma>0$, we have
 \begin{equation}\label{B.9}
   \begin{split}
     \|\upsilon_{\varepsilon}\|^2_{\varepsilon}=O(\int_{\R^N}\big[\varepsilon^2|\nabla w_{\varepsilon}(x)|^2+V(x)w_{\varepsilon}^2(x)
     -(p-1)W_{x,\varepsilon}^{p-2}w_{\varepsilon}^2(x)\big]\mathrm{d}x+\sum_{j=1}^k\alpha_{j,\varepsilon}^2\varepsilon^N
     +\varepsilon^{\gamma}\big).
   \end{split}
 \end{equation}
Also $\upsilon_\varepsilon(x)\in \bigcap^k_{j=1}E_{\varepsilon,a_j,x_{j,\varepsilon}}$,  Lemma \ref{lem-A.1} and \eqref{B.9}
 imply that, for any $\gamma>0$,
\begin{equation}\label{B.10}
  \begin{split}
    \|w_{\varepsilon}\|^2_{\varepsilon}&
    =\|\upsilon_{\varepsilon}\|^2_{\varepsilon}+\sum_{j=1}^k \| \alpha_{j,\varepsilon}U_{a_j}(\frac{x-x_{j,\varepsilon}}{\varepsilon})\|^2_{\varepsilon}+
    2\big(\upsilon_{\varepsilon}(x), \sum_{j=1}^k \alpha_{j,\varepsilon}U_{a_j}(\frac{x-x_{j,\varepsilon}}{\varepsilon})\big)_{\varepsilon}+I_1
   \\&
    \leq C\int_{\R^N}\big[\varepsilon^2|\nabla w_{\varepsilon}(x)|^2+V(x)w_{\varepsilon}^2(x)
    -(p-1)\displaystyle\sum_{j=1}^kU_{a_j}^{p-2}
    (\frac{x-x_{j,\varepsilon}}{\varepsilon})w_{\varepsilon}^2(x)\big]
    \mathrm{d}x\\&
    ~~~~+C\big(\varepsilon^{\gamma}+\sum_{j=1}^k\alpha_{j,\varepsilon}^2\varepsilon^N
    \big)+A_3,
   \end{split}
 \end{equation}
where
\begin{equation*}
\begin{split}
A_3&=
\sum_{j=1}^k \sum_{i\neq j}\int_{\R^N} \alpha_{j,\varepsilon}\alpha_{i,\varepsilon}
\big(U_{a_j}
    (\frac{x-x_{j,\varepsilon}}{\varepsilon}),U_{a_i}
    (\frac{x-x_{i,\varepsilon}}{\varepsilon})\big)_{\varepsilon}.
\end{split}
\end{equation*}
Also, \eqref{A.3} implies that,  for any $\gamma>0$,
\begin{equation}\label{B.11}
\begin{split}
A_3=O(\varepsilon^{\gamma}).
\end{split}
\end{equation}
 Then \eqref{2.2}, \eqref{A.3}, \eqref{A.4}, \eqref{B.10} and \eqref{B.11} imply, for any $\gamma >0$,
   \begin{equation}\label{B.12}
   \begin{split}
      \|w_{\varepsilon}\|^2_{\varepsilon}&
    =O\big(\int_{\R^N}\big[\varepsilon^2|\nabla w_{\varepsilon}(x)|^2+V(x)w_{\varepsilon}^2(x)
    -(p-1)\displaystyle\sum_{j=1}^kU_{a_j}^{p-2}
    (\frac{x-x_{j,\varepsilon}}{\varepsilon})w_{\varepsilon}^2(x)
    \big]\mathrm{d}x\big)\\&
    ~~~~+O(\sum_{j=1}^k\alpha_{j,\varepsilon}^2\varepsilon^N
    +\varepsilon^{\gamma}).
    \end{split}
 \end{equation}
On the other hand, from \eqref{B.1}, we know
 \begin{equation}\label{B.13}
 \begin{split}
   \int_{\R^N}&\big[\varepsilon^2|\nabla w_{\varepsilon}(x)|^2+V(x)w_{\varepsilon}^2(x)
    -(p-1)\displaystyle\sum_{j=1}^kU_{a_j}^{p-2}
    (\frac{x-x_{j,\varepsilon}}{\varepsilon})w_{\varepsilon}^2(x)\big]\mathrm{d}x\\&
    =\int_{\R^N}\big[N\big(w_{\varepsilon}(x)\big)w_{\varepsilon}(x)+l_{\varepsilon}(x)w_{\varepsilon}(x)\big]
   \mathrm{d}x.
   \end{split}
 \end{equation}
Then combining \eqref{2.4}, \eqref{A.3} and \eqref{A.11}, we can obtain
 \begin{equation}\label{B.14}
   \int_{\R^N}N\big(w_{\varepsilon}(x)\big)w_{\varepsilon}(x)\mathrm{d}x\leq C\big(\int_{\R^N}|w_{\varepsilon}(x)|^{p^*}\mathrm{d}x+\varepsilon^{\gamma}\big)= o(1)\|w_{\varepsilon}\|_{\varepsilon}^2+O(\varepsilon^{\gamma}),
 \end{equation}
where $p^*=\min\{p,3\}$. Also, \eqref{A.5} imply that
\begin{equation}\label{B.15}
\displaystyle\int_{\R^{N}}l_{\varepsilon}(x)w_{\varepsilon}(x)\mathrm{d}x=O\big(\varepsilon^{\frac{N}{2}+m}+\varepsilon^
   {\frac{N}{2}}\max_{j=1,\cdots,k.}|x_{j,\varepsilon}-a_j|^m\big)\|w_{\varepsilon}\|_{\varepsilon}.
\end{equation}
Then taking suitable $\gamma>0$, by \eqref{B.12}, \eqref{B.13}, \eqref{B.14} and \eqref{B.15}, we get \eqref{B.3}.
Also,
\eqref{B.3}, \eqref{B.12} and \eqref{B.13} imply \eqref{B.4}.
\end{proof}

\smallskip

{\bf Acknowledgements:}
{\it Daomin Cao was partially supported by Beijing Center
for Mathematics and Information Interdisciplinary Sciences and NSFC grants (No.11271354
and No.11331010).}

\renewcommand\refname{References}
\renewenvironment{thebibliography}[1]{%
\section*{\refname}
\list{{\arabic{enumi}}}{\def\makelabel##1{\hss{##1}}\topsep=0mm
\parsep=0mm
\partopsep=0mm\itemsep=0mm
\labelsep=1ex\itemindent=0mm
\settowidth\labelwidth{\small[#1]}%
\leftmargin\labelwidth \advance\leftmargin\labelsep
\advance\leftmargin -\itemindent
\usecounter{enumi}}\small
\def\newblock{\ }
\sloppy\clubpenalty4000\widowpenalty4000
\sfcode`\.=1000\relax}{\endlist}
\bibliographystyle{model1b-num-names}

\end{document}